\documentclass[12pt]{article}

\usepackage{latexsym}
\usepackage{amsopn,amscd}
\usepackage{amsfonts}
\usepackage{amssymb}
\usepackage{amsthm}
\usepackage{amsmath}
\usepackage{a4wide}
\usepackage[square,authoryear]{natbib}

\newtheorem{Theorem}{Theorem}[section]
\newtheorem{Proposition}[Theorem]{Proposition}
\newtheorem{Lemma}[Theorem]{Lemma}

\newtheorem{Remark}[Theorem]{Remark}

\def\id{\mathrm{Id}}

\def\fra{\mathfrak}
\def\Sigm{\mathcal S}

\def\HH{\widetilde {\mathrm H}}
\def\HHH{M}
\def\HHHH{\mathrm H}
\def\jj{j}
\def\JJJ{\mathrm J}
\def\rbar{{\mathrm B}}
\def\III{\mathrm I}
\def\ppartial{\mathcal D}
\def\dd{d}
\def\fra{\mathfrak}
\def\scs{\sc}
\def\Sigm{\mathrm S}
\def\SSS{\mathrm S^{\mathrm c}}
\def\rcob{\Omega}
\def\ppppartial{\mathcal D}
\def\dell{\mathcal D}
\def\ppi{\pi}
\def\Nsddata#1#2#3#4#5{
\begin{equation*}
   (#4
     \begin{CD}
      \null @>#2>> \null\\[-3.2ex]
      \null @<<#3< \null
     \end{CD}
    #1, #5)
  \end{equation*}
}

\numberwithin{equation}{section}

\begin{document}

\title{\bf On the construction of $A_{\infty}$-structures}

\author{
J.~Huebschmann
\\[0.3cm]
 USTL, UFR de Math\'ematiques\\
CNRS-UMR 8524
\\
59655 Villeneuve d'Ascq C\'edex, France\\
Johannes.Huebschmann@math.univ-lille1.fr
 }

\maketitle

\centerline{\it To Tornike Kadeishvili} 

\medskip
\begin{abstract}
{We relate a construction of Kadeishvili's
establishing an  $A_{\infty}$-structure on the homology of
a differential graded algebra 
or more generally  of an $A_{\infty}$ algebra
with certain constructions of Chen and Gugenheim.
Thereafter we establish the links of these 
constructions with subsequent developments.}
\end{abstract}

\noindent
\\[0.3cm]
\\[0.3cm]

\tableofcontents

\section{Introduction}

We will relate a construction of Kadeishvili's
establishing an  $A_{\infty}$-structure on the homology of
a differential graded algebra 
or more generally  of an $A_{\infty}$ algebra
with certain constructions of Chen and Gugenheim.
We will then establish the links of these 
constructions with subsequent developments.

Let $R$ be a commutative ring and $A$ a differential graded algebra
over $R$. Suppose that, as a graded $R$-module, the 
homology $\mathrm H (A)$ of $A$ is free.
Then  $\mathrm H (A)$ acquires an  $A_{\infty}$-algebra structure
that is equivalent to $A$. 
Over a field, 
so that the freeness hypothesis relative to 
$\mathrm H (A)$ is automatically satisfied,
this fact is nowadays quoted as the \lq\lq minimality theorem\rq\rq\ 
for differential graded algebras---
we will discuss the issue of \lq\lq minimality theorem\rq\rq\ below.
Such a result was published
by Kadeishvili in 1980 \cite{kadeione}.
Over a general ground ring $R$,
a related result involving HPT was
published by V. Gugenheim in 1982 \cite{gugenhth}.
More precisely, starting from a simply connected
coaugmented differential graded coalgebra $C$ over the ground ring $R$
that is homology split (e.~g. free as a module over the ground ring),
the homology $\HHHH(C)$ acquires a coaugmented graded coalgebra
structure, and a perturbation of the ordinary cobar construction
relative to the coalgebra structure on 
$\HHHH(C)$ yields an $A_{\infty}$-coalgebra structure
on $\HHHH(C)$ that is equivalent to the original
coalgebra $C$.
This is a version of the 
\lq\lq minimality theorem\rq\rq\ in the realm of coalgebras.
Gugenheim's approach relies on a perturbation argument
developed
over the reals by Chen \cite{chenfou}, published in 1977 and,
furthermore,
in a sense,
Theorem 3.1.1 in Chen's paper \cite{chen} establishes
a version of the \lq\lq minimality theorem\rq\rq.
In 1982, Kadeishvili published a result
which extends his original approach to
a more general
\lq\lq minimality theorem\rq\rq\ 
saying that, over a field, 
the homology  $\HHHH(A)$ of a general $A_{\infty}$-algebra $A$
acquires an  $A_{\infty}$-algebra structure
that is equivalent to $A$ 
\cite{kadeifiv}. 

Because of the present renewed interest in the
\lq\lq minimality theorem\rq\rq,
and to help the presently young avoid loss of contact with the past,
it seems worthwhile explaining the original insight
into the \lq\lq minimality theorem\rq\rq. 
This will place the original results properly
in the literature.
The contributions of Gugenheim and in particular those of Chen
seem to have been largely forgotten. Also there has been a debate in the
literature to what extent the various constructions of 
$A_{\infty}$-structures were explicit; the constructions by Chen,
Gugenheim and Kadeishvili are perfectly explicit.
In fact, we hope to convince the reader that these constructions
essentially boil down to the very same basic idea.
We will then relate the old approaches to subsequent  ones
and show that the recent ones \cite{kontsotw} (6.4), \cite{merkutwo}
essentially still come down to the same basic idea.
In particular we will illustrate how the constructions in terms
of labelled oriented rooted trees \cite{kontsotw} (6.4) are instances of
ordinary HPT constructions. This will, perhaps, demystify
the labelled oriented rooted trees method and make it accessible
to a wider audience.
We will also explain  the Lie algebra 
case.
It turns out that the various constructions
establishing the corresponding statement of the
\lq\lq minimality theorem\rq\rq\ 
in the algebra, coalgebra, Lie algebra situation, etc.
all boil down to essentially the same kind of construction,
as comparison
of  
\eqref{contco}--\eqref{shcc}, 
\eqref{contalg}--\eqref{sha}, 
and \eqref{contlie}--\eqref{shl} 
below shows.

It is a pleasure to dedicate this paper to Tornike Kadeishvili.
I am indebted to him for collaboration and for discussion much beyond
that collaboration. 
Our collaboration \cite{huebkade} led to a number of results related with
the {\em perturbation lemma\/}; in particular we have elaborated on the
compatibility of the perturbation lemma with suitable algebraic structure.
The perturbation lemma  is lurking behind the formulas in
Chapter II of Section 1 of \cite{shih} and seems to have first been
made explicit by M. Barrat (unpublished). The first instance known
to us where it appeared in print is \cite{rbrown}.
By means of that lemma, in \cite{gugenhtw},
V. Gugenheim  developed a lucid proof of the twisted Eilenberg-Zilber theorem
which, in turn, 
was established by E. H. Brown \cite{ebrown} originally via acyclic models.
We have already mentioned Chen's construction  of a perturbation
given in \cite{chen}, 
extended and clarified by V. Gugenheim in \cite{gugenhth}.
These constructions of Chen's and Gugenheim's are somewhat by hand
and do {\em not\/} involve the perturbation lemma.
In \cite{gugensta}, Gugenheim and Stasheff extended that 
construction of a perturbation to the case where the contracted object 
is admitted to have non-zero differential.
In \cite{homotype}, 
I had developed the tensor trick,
see Section \ref{algebra} below,
and the idea of iterative perturbation.
My collaboration with T. Kadeishvili
involved the tensor trick and iterative perturbations and
produced in particular
the {\em (co)algebra perturbation lemma\/}.
This lemma then enabled us to recover the
perturbations of the kind constructed by Chen, Gugenheim, and 
Gugenheim-Stasheff in a conceptual way.
It led as well to a lucid proof of the minimality theorem.
We also developed a perturbation theory
for a general homotopy equivalence,
not necessarily a contraction.
It is, furthermore,  worthwhile noting that
the {\em labelled rooted trees
are lurking behind
the (co)algebra perturbation lemma\/}
but at the time there was no need to spell them out explicitly
in \cite{huebkade}.
We will explain all these and more issues 
in this paper. At these days the perturbation lemma
and variants thereof
are, perhaps, more vivid than ever, see e.~g. \cite{aafr}, \cite{pertlie},
\cite{pertlie2}, \cite{marklthr} and the references there;
in particular, the perturbation theory
for a general homotopy equivalence 
developed in \cite{huebkade}
has been taken up again and pushed
further in \cite{marklthr}. 

I owe some special thanks to Jim Stasheff
for a number of comments on a draft of this paper.

\section{Preliminaries}

The ground ring is a commutative ring with $1$ and will be denoted
by $R$. Later some condition has, perhaps,
 to be imposed upon $R$ 
so that the symmetric coalgebra on the $R$-module under discussion exists
but $R$
is not necessarily a field.  

Indeed, to avoid confusion, recall that, given the graded $R$-module $Y$,
for $j \geq 0$, the notation $\Sigm_j^{\mathrm
c}[Y]\subseteq \mathrm T_j^{\mathrm c}[Y]$  refers to the submodule of
invariants in the $j$'th tensor power $\mathrm T_j^{\mathrm c}[Y]$
relative to the obvious action on $\mathrm T_j^{\mathrm c}[Y]$ of
the symmetric group $S_j$ on $j$ letters, and $\Sigm^{\mathrm
c}[Y]$ refers to the direct sum
\[ \Sigm^{\mathrm
c}[Y]=\oplus_{j=0}^{\infty} \Sigm_j^{\mathrm c}[Y]
\]
of graded $R$-modules.
Some hypothesis is, in general, necessary in order for
the homogeneous constituents
\[
\mathrm T_{j+k}^{\mathrm c}[Y] \longrightarrow \mathrm
T_j^{\mathrm c}[Y] \otimes \mathrm T_k^{\mathrm c}[Y]\ (j,k \geq
0)
\]
of the diagonal map $\Delta \colon \mathrm T^{\mathrm c}[Y] \to
\mathrm T^{\mathrm c}[Y] \otimes \mathrm T^{\mathrm c}[Y]$ of the
graded tensor coalgebra $\mathrm T^{\mathrm c}[Y]$ to induce a
graded diagonal map on $\Sigm^{\mathrm c}[Y]$,
so that $\Sigm^{\mathrm c}[Y]$ is then the {\em symmetric coalgebra\/}
on $Y$.
See Section 3 of \cite{pertlie}.
In particular, 
let $V$ be a projective graded $R$-module, concentrated in odd degrees,
and consider the graded exterior algebra $\Lambda[V]$ on $V$.
The diagonal map $V\to V \oplus V$ is well known to induce a diagonal
map for $\Lambda[V]$ turning the latter into a graded Hopf algebra.
We then denote the resulting graded coalgebra by
$\Lambda'[V]$ and, as usual, refer to it as the {\em exterior coalgebra\/}.
Whenever a graded exterior coalgebra of the kind  $\Lambda'[V]$
is under discussion,
we will suppose throughout that the resulting coalgebra is the
graded symmetric coalgebra $\SSS [V]$ on $V$, that is, that the
canonical morphism of coalgebras from 
$\Lambda'[V]$ to $\SSS [V]$ (induced by the canonical projection
from $\Lambda'[V]$ to $V$)
is an isomorphism of graded coalgebras.
This excludes the prime 2 being a zero divisor in the ground ring $R$.
In particular, a field of characteristic 2 is not admitted as ground ring.
Indeed in characteristic 2 the entire theory requires special treatment.

We will take {\em chain complex\/} to mean {\em
differential graded\/} $R$-{\em module\/}. A chain complex will
not necessarily be concentrated in non-negative or non-positive
degrees. The differential of a chain complex will always be
supposed to be of degree $-1$. For a filtered chain complex $X$, a
{\em perturbation\/} of the differential $d$ of $X$ is a
(homogeneous) morphism $\partial$ of the same degree as $d$ such
that $\partial$ lowers the filtration and $(d +
\partial)^2 = 0$ or, equivalently,
\begin{equation}
[d,\partial] + \partial \partial = 0.
\end{equation}
Thus, when $\partial$ is a perturbation on $X$, the sum $d +
\partial$, referred to as the {\em perturbed differential\/},
endows $X$ with a new differential. When $X$ has a graded
coalgebra structure such that $(X,d)$ is a differential graded
coalgebra, and when the {\em perturbed differential\/} $d +
\partial$ is compatible with the graded coalgebra structure, we
refer to $\partial$ as a {\em coalgebra perturbation\/}; the
notion of {\em algebra perturbation\/} is defined similarly. Given
a differential graded coalgebra $C$ and a coalgebra perturbation
$\partial$ of the differential $d$ on $C$, we will occasionally
denote the new or {\em perturbed\/} differential graded coalgebra
by $C_{\partial}$.
Likewise given a differential graded algebra $A$ and an algebra perturbation
$\partial$ of the differential $d$ on $A$, we will occasionally
denote the new or {\em perturbed\/} differential graded algebra
by $A_{\partial}$.

A {\em contraction\/}
\begin{equation}
   (N
     \begin{CD}
      \null @>{\nabla}>> \null\\[-3.2ex]
      \null @<<{\pi}< \null
     \end{CD}
    M, h) \label{co}
  \end{equation}
of chain complexes \cite{eilmactw} consists of

-- chain complexes $N$ and $M$,
\newline
\indent -- chain maps $\pi\colon N \to M$ and $\nabla \colon M \to
N$,
\newline
\indent --  a morphism $h\colon N \to N$ of the underlying graded
modules of degree 1;
\newline
\noindent these data are required to satisfy
\begin{align}
 \pi \nabla &= \mathrm{Id},
\label{co0}
\\
Dh &= \mathrm{Id} -\nabla \pi, \label{co1}
\\
\pi h &= 0, \quad h \nabla = 0,\quad hh = 0. \label{side}
\end{align}
The requirements \eqref{side} are referred to as {\em annihilation
properties\/} or {\em side conditions\/}.

Let $C$ be a {\em coaugmented\/} differential graded coalgebra
with coaugmentation map $\eta \colon R \to C$ and {\em
coaugmentation\/} coideal $\JJJ C = \mathrm{coker}(\eta)$, the
diagonal map being written as $\Delta \colon C \to C \otimes C$ as
usual. Recall that the counit $\varepsilon \colon C \to R$ and the
coaugmentation map determine a direct sum decomposition $C = R
\oplus \JJJ C$. The {\em coaugmentation\/} filtration $\{\mathrm
F_nC\}_{n \geq 0}$ is as usual given by
\[\mathrm F_nC = \mathrm{ker}(C \longrightarrow (\JJJ C)^{\otimes
(n+1)})\  (n \geq 0)
\]
where the unlabelled arrow is induced by some iterate of the
diagonal $\Delta$ of $C$. This filtration is well known to turn
$C$ into a {\em filtered\/} coaugmented differential graded
coalgebra; thus, in particular, $\mathrm F_0C = R$. We recall that
$C$ is said to be {\em cocomplete\/} when $C=\cup \mathrm F_nC$.

Write $s$ for the {\em suspension\/} operator as usual and
accordingly $s^{-1}$ for the {\em desuspension\/} operator. Thus,
given the chain complex $X$, $(sX)_j = X_{j-1}$, etc., and the
differential ${d\colon sX \to sX}$ on the suspended object $sX$ is
defined in the standard manner so that $ds+sd=0$.

Given two chain complexes $X$ and $Y$, recall that
$\mathrm{Hom}(X,Y)$ inherits the structure of a chain complex by
the operator $D$ defined by
\begin{equation}
D \phi = d \phi -(-1)^{|\phi|} \phi d
\end{equation}
where $\phi$ is a homogeneous homomorphism from $X$ to $Y$ and
where $|\phi|$ refers to the degree of $\phi$.

Let $\fra g$ be a chain complex having the property that
the cofree coaugmented differential graded cocommutative
coalgebra $\Sigm^{\mathrm c}[s\fra g]$ on the suspension $s\fra g$
of $\fra g$ exists.
This happens to be the case, e.g.,  when $\fra g$ is projective
as a graded $R$-module.
Let
\[
\tau_{\fra g}\colon \Sigm^{\mathrm c}[s\fra g]\longrightarrow \fra
g
\]
be the composite of the canonical projection to $\Sigm_1^{\mathrm
c}[s\fra g] =s\fra g$ with the desuspension map. Suppose that
$\fra g$ is endowed with a graded skew-symmetric bracket $[\,
\cdot \, , \, \cdot \, ]$ that is compatible with the differential
but not necessarily a graded Lie bracket, i.~e. does not
necessarily satisfy the graded Jacobi identity. Let $C$ be a
coaugmented differential graded cocommutative coalgebra. Given
homogeneous morphisms $a,b \colon C \to \fra g$, with a slight
abuse of the bracket notation $[\, \cdot \, , \, \cdot \, ]$, the
{\em cup bracket\/} $[a, b]$ is given by the composite
\begin{equation*}
\begin{CD} C @>{\Delta}>> C\otimes C @>{a\otimes b}>>\fra g
\otimes\fra g @> {[\cdot,\cdot]}>> \fra g.
\end{CD}
\end{equation*}
The cup bracket $[\, \cdot \, , \, \cdot \, ]$ is well known to be
a graded skew-symmetric bracket on $\mathrm{Hom}(C,\fra g)$ which
is compatible with the differential on $\mathrm{Hom}(C,\fra g)$.
Define the coderivation
\[
\partial\colon\Sigm^{\mathrm
c}[s\fra g] \longrightarrow \Sigm^{\mathrm c}[s\fra g]
\]
on $\Sigm^{\mathrm c}[s\fra g]$  by the requirement
\begin{equation}
\tau_{\fra g} \partial = \tfrac 12 [\tau_{\fra g},\tau_{\fra g}]
\colon \Sigm_2^{\mathrm c}[s\fra g] \to \fra g. \label{proc0333}
\end{equation}
Then $D\partial\  (=d\partial + \partial d) = 0$ since the bracket
on $\fra g$ is supposed to be compatible with the differential
$d$. Moreover, {\sl the bracket on $\fra g$ satisfies the graded
Jacobi identity if and only if $\partial\partial = 0$, that is, if
and only if $\partial$ is a coalgebra perturbation of the
differential $d$ on\/} $\Sigm^{\mathrm c}[s\fra g]$, cf. e.~g.
\cite{huebstas}.

We now suppose that the graded bracket $[\, \cdot \, , \, \cdot \,
]$ on $\fra g$ turns $\fra g$  into a differential graded Lie
algebra and continue to denote the resulting coalgebra
perturbation by $\partial$, so that $\Sigm_{\partial}^{\mathrm
c}[s\fra g]$ is a coaugmented differential graded cocommutative
coalgebra; in fact, $\Sigm_{\partial}^{\mathrm c}[s\fra g]$ is
then precisely the ordinary {\scs
C(artan-)C(hevalley-)E(ilenberg)\/} or {\em classifying\/}
coalgebra for $\fra g$ and, following \cite{quilltwo} (p.~291), we
denote it by $\mathcal C[\fra g]$ (but the construction given
above is different from that in \cite{quilltwo} which, in turn, is
carried out only over a field of characteristic zero).
Furthermore, given a coaugmented differential graded cocommutative
coalgebra $C$, the cup bracket turns $\mathrm{Hom}(C,\fra g)$ into
a differential graded Lie algebra. In particular,
$\mathrm{Hom}(\Sigm^{\mathrm c},\fra g)$ and $\mathrm{Hom}(\mathrm
F_n\Sigm^{\mathrm c},\fra g)$ ($n\geq 0$) acquire differential
graded Lie algebra structures.

Given a  coaugmented differential graded cocommutative coalgebra
$C$ and a differential graded Lie algebra $\fra h$, a {\em Lie
algebra twisting cochain\/} $t \colon C \to \fra h$ is a
homogeneous morphism of degree $-1$ whose composite with the
coaugmentation  map is zero and which satisfies
\begin{equation}
Dt = \tfrac 12 [t,t] , \label{master}
\end{equation}
cf. \cite{moorefiv}, \cite{quilltwo}. In particular, relative to the
graded Lie bracket $\ppppartial$ on $\fra g$, the morphism
$\tau_{\fra g}\colon
\mathcal C[\fra g] \to \fra g$ is a Lie algebra twisting cochain,
the {\scs C(artan-)C(hevalley-)E(ilenberg)\/} or {\em universal\/}
Lie algebra twisting cochain for $\fra g$. It is, perhaps, worth
noting that, when $\fra g$ is viewed as an abelian differential
graded Lie algebra relative to the zero bracket, $\Sigm^{\mathrm
c}[s\fra g]$ is the corresponding {\scs CCE\/} or {\em
classifying\/} coalgebra and $\tau_{\fra g}\colon \Sigm^{\mathrm
c}[s\fra g] \to \fra g$ is still the universal Lie algebra
twisting cochain.

\begin{Remark} The terms
{\em master equation\/},
{\em Maurer-Cartan equation\/},
{\em Lie algebra twisting cochain\/}, and
{\em integrability condition\/}
all refer to the same mathematical object.
Historically the Maurer-Cartan equation came first.
\end{Remark}

At the risk of making a mountain out of a molehill, we note that,
in \eqref{proc0333} and \eqref{master} above, the factor $\tfrac
12$ is a mere matter of convenience. The correct way of phrasing
graded Lie algebras when the prime 2 is not invertible in the
ground ring is in terms of an additional operation, the {\em
squaring\/} operation $\mathrm {Sq}\colon \fra g_{\mathrm{odd}}
\to \fra g_{\mathrm{even}}$ and, by means of this operation, the
factor $\tfrac 12$ can be avoided. Indeed, in terms of this
operation, the equation \eqref{master} takes the form
\[
Dt=\mathrm{Sq}(t).
\]
For intelligibility, we will follow the standard convention, avoid
spelling out the squaring operation explicitly, and keep the
factor $\tfrac 12$. A detailed description of the requisite
modifications when the prime 2 is not invertible in the ground ring is
given in \cite{pertlie2}.

Finally we comment on the usage of the terminology \lq\lq minimal\rq\rq:
Given an augmented differential graded algebra $A$ over a field $k$,
with augmentation map $\varepsilon \colon A \to k$
and augmentation ideal $\III A$,
its graded vector space
 of {\em indecomposables\/} $Q(A)$ is the cokernel of the canonical map
$\III A \otimes \III A \to \III A$ induced by the multiplication map of $A$;
since $A$ is a differential graded algebra, this cokernel inherits a
differential, and the augmented differential graded algebra $A$ is said to be
{\em minimal\/} when it is cofibrant and when
the differential on the indecomposables $Q(A)$
is zero. Any connected differential graded algebra has a canonical 
augmentation.
In rational homotopy theory, a {\em minimal model\/} of
a connected differential graded commutative algebra $A$
is a minimal differential graded commutative algebra $mA$ together with
a morphism $mA \to A$ of differential graded algebras
which is an isomorphism on homology.
Likewise, 
the differential graded Lie algebra $L$ is said to be
{\em minimal\/} when it is cofibrant and when
the differential on the abelianization $Q(L)$
is zero. A {\em minimal model\/} of
a connected differential graded Lie algebra $L$
is a minimal differential graded Lie algebra $mL$ together with
a morphism $mL \to L$ of differential graded Lie algebras
which is an isomorphism on homology.
There are also corresponding notions of minimal differential graded
coalgebra and of minimal model for a differential graded coalgebra.
See e.~g. \cite{neiseone} (Section 5) for details and more references. 
Over a local ring $R$, with maximal ideal $\mathfrak m \subseteq R$,
a free resolution
\begin{equation*}
\begin{CD}
0 @<<< M  @<{\varepsilon}<< F_0  @<d<< F_1  @<d<<
\end{CD}
\end{equation*}
is said to be minimal when $d(F_j) \subseteq \mathfrak m F_{j-1}$
for $j \geq 1$.
The meaning of the term \lq\lq minimal\rq\rq\ in the present paper 
refers to $A_{\infty}$-algebras, see Section \ref{infalg} below.
While in rational homotopy theory, 
a homology algebra is not necessarily a Sullivan algebra,
suitably interpreted, the notion of minimality
of $A_{\infty}$-algebras
is consistent with the usage of the concept of minimality
in rational homotopy theory.
See Remarks \ref{differ} and \ref{chenmini} below.

\section{$A_{\infty}$-algebras}
\label{infalg}

To introduce language and notation we reproduce a precise definition of
an  $A_{\infty}$-algebra and of a morphism of
 $A_\infty$-algebras, cf. \cite{stashone}. 
Our convention is that a differential 
lowers degree by 1.

An $A_\infty$-algebra over the ground ring $R$ is a graded
$R$-module $A$ equipped with a family
$\{m_n\}_{n=1}^\infty$ of $R$-multilinear maps 
\[
m_n \colon A^{\otimes n}\longrightarrow A
\]
of degree $n-2$ that satisfy the identities
\begin{equation}
 \sum_{r+s+t=n} (-1)^{r+st} m_{r+1+t}(\id^r\otimes m_s\otimes \id^t)=0
\label{Stn}
\end{equation}

A morphism $f\colon A\to B$ of $A_\infty$-algebras is a family
$\{f_n\}_{n=1}^\infty$ of $R$-multilinear maps 
\[
f_n\colon A^{\otimes n}\longrightarrow B
\]
of degree $n-1$ that satisfy the identities
\begin{equation}
\sum_{r+s+t=n}(-1)^{r+st}f_{r+1+t}(\id^r\otimes m_s\otimes\id^t) =
\sum(-1)^wm_q(f_{i_1}\otimes\dots\otimes f_{i_q})
\label{Stmn}
\end{equation}
where $i_1+\dots+i_q=n$ and
$w=(q-1)(i_1-1)+\dots+2(i_{q-2}-1)+(i_{q-1}-1)$.

We now reproduce the familiar equivalent 
description of an $A_{\infty}$-algebra 
structure as a coalgebra perturbation.
To save trouble, we will do this only for the {\em supplemented\/} case.
Thus let $\HHH$ be a  
graded $R$-module
which comes with a direct sum decomposition
$\HHH = \III \HHH \oplus R$ of 
graded $R$-modules.
We view $\HHH$ as a 
graded algebra
with zero multiplication on $\III\HHH$,
so that $\III\HHH$ can then be interpreted as 
the augmentation ideal of $\HHH$.
The 
graded {\em tensor\/} coalgebra
$\mathrm T^{\mathrm c}[s\III\HHH]$
(with zero differential)
is then the corresponding reduced {\em bar\/} construction for $\HHH$;
let $\tau_{\HHH}\colon \mathrm T^{\mathrm c}[s\III\HHH]\to \HHH$
be the universal bar construction
twisting cochain, that is, the 
canonical projection to $s\III\HHH$,
followed by the desuspension mapping.

For $j \geq 1$, 
let
\[
m_j\colon (\III\HHH)^{\otimes j}
\longrightarrow \III\HHH
\]
be a homogeneous degree $j-2$ operation 
and define the coderivation 
\[
\mathcal D^{j-1} \colon
 \mathrm T^{\mathrm c}[s\III\HHH]
\longrightarrow \mathrm T^{\mathrm c}[s\III\HHH]
\]
by the identity
\[
m_j=s^{-1}\circ \mathcal D^{j-1}\circ s^{\otimes j}
\colon (\III\HHH)^{\otimes j}
\longrightarrow \III\HHH.
\]
For convenience, we write $d= \mathcal D^0$.
Then
\[
\mathcal D=\mathcal D^1+ \mathcal D^2 + \ldots\colon
 \mathrm T^{\mathrm c}[s\III\HHH]
\longrightarrow \mathrm T^{\mathrm c}[s\III\HHH]
\]
is a coderivation, and so is the sum
\[
d+\mathcal D=d+\mathcal D^1+ \mathcal D^2 + \ldots\colon
 \mathrm T^{\mathrm c}[s\III\HHH]
\longrightarrow \mathrm T^{\mathrm c}[s\III\HHH].
\]

\begin{Proposition} {\rm (i)}
The family $\{m_j\}_j$ turns $\HHH$ into an (augmented)
$A_{\infty}$-algebra if and only if
$d+\mathcal D$ is  a differential, that is, if and only if
\[
d\mathcal D+ \mathcal D d + \mathcal D\mathcal D=0. 
\]
{\rm (ii)} Given the augmented $A_{\infty}$-algebras $A$ and $B$,
a family
$\{f_j\}_j$ of $R$-multilinear maps 
\[
f_j\colon (\III A)^{\otimes n}\longrightarrow \III B
\]
of degree $j-1$ 
is a morphism of (augmented)
$A_{\infty}$-algebras if and only if the constituents of the family
combine to a morphism
\[
(\mathrm T^{\mathrm c}[s\III A],d +\mathcal D)
\longrightarrow
(\mathrm T^{\mathrm c}[s\III B],d +\mathcal D)
\]
of differential graded coalgebras.
\end{Proposition}

Thus, when  $\{m_j\}_j$ turns $\HHH$ into an (augmented)
$A_{\infty}$-algebra,
in particular, $m_1$ is a differential on $\HHH$ and on $\III \HHH$,
and $d$ is a differential on 
$\mathrm T^{\mathrm c}[s\III\HHH]$,
in fact, the differential induced by that on $\III \HHH$.
The special case where only $m_1$ and $m_2$ are non-zero
is that of an ordinary differential graded algebra structure,
and 
$(\mathrm T^{\mathrm c}[s\III\HHH],d +\mathcal D)$
is then the ordinary reduced bar construction
$\rbar \HHH$.

Kadeishvili introduced the terminology {\em minimal\/}
for an $A_\infty$-algebra $(M,\{m_i\})$ 
having  $m_1$ zero, i.e. trivial differential.  
Minimal $A_\infty$-algebras behave
similar to Sullivan's minimal differential graded 
algebras: each weak equivalence of
minimal  $A_\infty$-algebras is an isomorphism.
See also Remark \ref{chenmini} below.

We will henceforth consider $\III \HHH$ as a chain complex,
use the notation
\begin{equation}
\rbar_{\mathcal D}\HHH=
(\mathrm T^{\mathrm c}[s\III\HHH], d + \mathcal D) ,
\label{shaa}
\end{equation}
and occasionally refer to $\rbar_{\mathcal D}\HHH$
as the {\em standard construction\/}.
It is also customary to write
$\widetilde \rbar\HHH$ and to refer to the
{\em bar tilde\/} construction.
Apart from Section \ref{kadeal} below,
we will henceforth exclusively use the description
of an $A_{\infty}$-algebra structures on the chain complex $\HHH$ in terms of 
the coalgebra perturbation $\mathcal D$ on 
the associated differential graded coalgebra
$(\mathrm T^{\mathrm c}[s\III\HHH], d)$.
Likewise we will exclusively
use the description
of an $A_{\infty}$-coalgebra structure in terms of 
the corresponding algebra perturbation on the associated
differential graded algebra and
we will use the description
of an $L_{\infty}$-algebra structure merely in terms of 
the corresponding coalgebra perturbation on the associated
differential graded coalgebra.

\section{Kadeishvili's minimality theorem for algebras}
\label{kadeal}

In \cite{kade}, Kadeishvili  studied the homology of 
a fiber bundle with structure group $G$ and fiber $F$.
He noticed  that the Pontrjagin ring structure of 
the homology $\mathrm H_*(G)$ and the action of 
$\mathrm H_*(G)$ on $\mathrm H_*(F)$
in general fail to recover the geometry of the action.
To fix this failure,
he then constructed  certain higher
operations
\[
f^i: \mathrm H_*(G)\otimes ... (i\ \mathrm{times}) ...\otimes \mathrm H_*(G)\to \mathrm H_*(G),\
i=3,4, \ldots,
\]
and homomorphisms
\[
A^i: \mathrm H_*(G)\otimes ... (i\ \mathrm{times}) ... \otimes \mathrm H_*(G)\to C_*(G),\
i=1,2,3,4, \ldots .
\]
which are solutions of certain equations. For example, $f=f^3+f^4+ \ldots$, 
interpreted as a Hochschild cochain in $C^*(\mathrm H^*(G),\mathrm H^*(G))$,
satisfies the condition
\[
\delta f=f\cup_1 f
\]
where the operation
\[
\cup_1 \colon (a,b) \mapsto a\cup_1b
\] 
refers to the operation in the Hochschild (cochain) complex
introduced by Gerstenhaber.
Kadeishvili referred to this operation as a \lq\lq cup-one\rq\rq\ product
since it has the same properties as Steenrod's $\cup_1$-product, 
and he called such an $f$ 
{\em Hochschild twisting cochain\/} (page 3 of \cite{kade}).

With hindsight we see that Kadeishvili's construction just explained
yields an
$A_\infty$-algebra structure  $(\mathrm H_*(G),\{f^i\})$ and a morphism
(weak equivalence) of $A_\infty$-algebras
$\{A^i\}:(\mathrm H_*(G),\{f^i\})\to C_*(G)$.
While, at the time of writing \cite{kade}, 
Kadeishvili did not know about Stasheff's notion of
$A_\infty$-algebra  he realized thereafter that the condition
$\delta f=f\cup_1 f$ is exactly Stasheff's defining condition for
an $A_\infty$-algebra $(M,\{m_i\})$ with $m_1=0$.
This led to the paper \cite{kadeione}.
Here is the main result thereof,
valid for a general differential graded algebra, not necessarily of the kind
$C_*(G)$ for a group $G$.

\begin{Theorem}[Minimality theorem]\label{kadea}
Let $A$ be a differential graded algebra
over a field. There is an $A_\infty$-algebra structure on
$\HHHH(A)$ and an $A_\infty$-algebra quasi-isomorphism $f\colon\HHHH(A)\to A$
such that $f_1$ is a cycle-choosing quasi-isomorphism of 
chain complexes where the differential $m_1$ 
on $\HHHH(A)$ is zero and $m_2$ is a
strictly associative multiplication
induced by the multiplication in $A$.
The resulting structure is unique up to quasi-isomorphism.
When  $A$ has a unit, then the structure and quasi-isomorphism can be
chosen to be strictly unital.
\end{Theorem}

In particular, the
$A_\infty$-algebra structure on homology resulting from the minimality
theorem is minimal and this structure, which is of
course  not uniquely determined, is unique up to isomorphism in
the category of $A_\infty$-algebras.

To compare the original approach with other developments,
we partly reproduce the proof in \cite{kadeione}.

\begin{proof}
 Since $A$ is a differential graded 
algebra, its associated
$A_\infty$-structure is encapsulated in  the operations
$m_1$ and $m_2$, the higher operations being zero. We shall denote $m_1$
by $d$ and refer to $m_2$ by the notation $\,\cdot\,$ 
or simply by juxtaposition.

To start an inductive construction of an $A_\infty$-structure on
$\HHHH(A)$, we pick $m_1=0$ and take $m_2$ to be
the induced strictly associative multiplication on $\HHHH(A)$.
Furthermore, we take $f_1$ to be some linear map
$\HHHH(A)\to A$ that picks a cycle in each homology class.

Given $a_1,a_2 \in A$, let 
\[
\Psi_2(a_1,a_2)=f_1(a_1a_2)-f_1(a_1)f_1(a_2).
\]
This yields a
boundary, since $f_1(a_1a_2)$ is defined to be a representative
cycle of the homology class containing $f_1(a_1)f_1(a_2)$. Hence,
$\Psi_2(a_1,a_2)=dw$ for some $w$.
Abstracting from the particular elements $a_1$ and $a_2$, since we
are over a field, we find a morphism $f_2$ such that
$df_2=\Psi_2$.

Now, let $n>2$. Given $a_1,\dots,a_n \in A$,
let 
  \begin{multline*}
    \Psi_n(a_1,\dots,a_n)=
    \sum_{s=1}^{n-1}(-1)^{\varepsilon_1(a_1,\dots,a_n,s)}
    f_s(a_1,\dots,a_s)\cdot f_{n-s}(a_{s+1},\dots,a_n) + \\
    \sum_{j=2}^{n-1}\sum_{k=0}^{n-j}(-1)^{\varepsilon_2(a_1,\dots,a_n,k,j)}
    f_{n-j+1}(a_1,\dots,a_k,m_j(a_{k+1},\dots,a_{k+j}),\dots,a_n)
  \end{multline*}
where the expressions
\begin{align*}
\varepsilon_1(a_1,\dots,a_n,s)&=s+(n-s+1)(|a_1|+\dots+|a_s|)\\
\varepsilon_2(a_1,\dots,a_n,k,j)&=k+j(n-k-j+|a_1|+\dots+|a_k|)
\end{align*}
are the signs in \eqref{Stmn}, adjusted according
to the Eilenberg-Koszul convention.
The term $\Psi_n$ arises from the identity
\eqref{Stmn}, with the two terms $f_1m_n$ and $m_1f_n$ removed. 
To complete the inductive step we must exhibit
suitable terms $m_n$ and $f_n$ in such a way that the identity
\eqref{Stmn} holds.

Tedious but straightforward calculation shows that the element
$\Psi_n(a_1,\dots,a_n)$ is a $d$-cycle, and we take  
  \[
   m_n(a_1,\dots,a_n)= [\Psi_n(a_1,\dots,a_n)] \in \HHHH(A) .
  \]
This yields the operation $m_n$ on $\HHHH(A)$.
Since now $f_1(m_n(a_1,\dots,a_n))$ and $\Psi_n(a_1,\dots,a_n)$
are in the same class, there is some  $w\in A$
such that 
\[
f_1(m_n(a_1,\dots,a_n))-\Psi_n(a_1,\dots,a_n)=dw. 
\]
Abstracting
from the particular elements $a_1,\ldots,a_n$, since we
are over a field, we find a morphism $f_n$ such that
\[
df_n=\Psi_n.
\]
Thus $m_n$ and $f_n$ match the definitions of the corresponding
constituents of an
$A_{\infty}$-algebra and of a morphism of $A_{\infty}$-algebras,
respectively.
\end{proof}

M.~Vejdemo-Johansson
has observed that this proof can be translated 
into an algorithm for the
computation of the $A_\infty$-structure maps \cite{vejjoh}.
This justifies the  claim made earlier that
Kadeishvili's construction can be made explicit
(by means of a choice of contracting homotopy, see
Remark \ref{kami1} below).

\begin{Remark}\label{differ} Let $A$ be a connected minimal differential graded
commutative algebra over the rationals.
Theorem {\rm \ref{kadea}} applies to it and yields a \lq\lq minimal\rq\rq\ 
$A_{\infty}$-structure on $\HHHH(A)$, associated with $A$.
However the two notions of minimality clearly differ
unless $A$ is formal.
\end{Remark}

In 1982, Kadeishvili extended this construction
and arrived at a more general
\lq\lq minimality theorem\rq\rq\ 
saying that, over a field, 
the homology  $\HHHH(A)$ of a general $A_{\infty}$-algebra $A$
acquires an  $A_{\infty}$-algebra structure
that is equivalent to $A$ 
\cite{kadeifiv}.

\begin{Remark}\label{opera} Kadeishvili's original problem, that is,
that of constructing a small model for the chains on the total space of a fiber bundle,
has received much attention in the literature, as has the problem,
given a group $G$, to isolate, under suitable circumstances, a suitable 
structure on the homology $\mathrm H(X)$ of a $G$-space $X$ such that $X$ and
$\mathrm H(X)$ are equivalent in the  $A_{\infty}$-sense.
In the de Rham setting, this problem was studied, e.g., in {\rm \cite{gorkomac}},
where the corresponding $A_{\infty}$-structure on de Rham cohomology
is encoded in terms of what are referred to there as {\em cohomology operations\/}.
In {\rm \cite{koszultw}}, we have explained how equivariant de Rham theory can be subsumed
under relative homological algebra. This includes an explanation of
those $A_{\infty}$-structures on de Rham cohomology.
\end{Remark}

In the situation of the minimality theorem, Theorem \ref{kadea} above, when
the differential graded algebra $A$ is graded commutative,
the resulting $A_\infty$-algebra structure on
$\HHHH(A)$ has peculiar features encoded by
Kadeishvili in the notion of $CA_\infty$-algebra, meaning that,
in this case, the bar tilde construction is not only 
a differential graded coalgebra but also
an algebra with respect to the shuffle product, 
and the two structures combine to that of a differential graded
bialgebra. Later this kind of structure has been christened
$C_\infty$-structure.
In the special situation where
$A$ is the algebra of rational cochains on a
space $X$, the resulting $C_{\infty}$-algebra
structure on $\mathrm H^*(X,Q)$  determines the
rational homotopy type of $X$.
Kadeishvili has worked this out in \cite{kadeisev};
an extended version can be found in \cite{kadeinin}.

The general situation is this: Given an augmented $C_\infty$-algebra $A$,
the structure being given in terms of its standard construction
$\rbar_{\partial}A$, by the very definition of $C_\infty$-structure,
the shuffle multiplication turns $\rbar_{\partial}A$
into a graded commutative differential graded Hopf algebra, the space of
indecomposables relative to the algebra structure is a {\em differential graded
Lie coalgebra\/} $\mathrm L^{\mathrm c}$, 
in fact, a {\em perturbation\/} of the {\em cofree differential graded
Lie coalgebra\/} $\mathrm L^{\mathrm c}(s\III A)$ {\em cogenerated by\/} $s\III A$, 
and the projection $\rbar_{\partial}A \to \mathrm L^{\mathrm c}$
actually spells out the differential graded coalgebra
$\rbar_{\partial}A$ as the {\em universal coalgebra\/}
$\mathrm U^{\mathrm c}[\mathrm L^{\mathrm c}]$ {\em cogenerated by\/} $\mathrm L^{\mathrm c}$,
the space $\mathrm L^{\mathrm c}$ necessarily being that of indecomposables relative to the
graded commutative algebra structure.
More formally, the {\em standard construction for the augmented\/}
$C_{\infty}$-{\em algebra $A$ boils down a perturbation of the 
cofree differential graded
Lie coalgebra $\mathrm L^{\mathrm c}(s\III A)$ cogenerated by\/} $s\III A$. 
In other words, the $C_{\infty}$-structure on $A$ is given by a 
{\em perturbation of the differential on the cofree differential graded
Lie coalgebra $\mathrm L^{\mathrm c}(s\III A)$ cogenerated by\/} $s\III A$. 

In the situation of Kadeishvili's observation  explained above,
the differential graded Lie coalgebra 
$\mathrm L^{\mathrm c}$ is the dual of the familiar minimal Lie algebra model in rational homotopy
theory. Indeed, the structure dual to that  of a $C_\infty$-structure
has been explored in the literature in the context of rational homotopy theory.
We will explain this in Remark
\ref{chenmini} below.

Lie coalgebras are still not widely known objects. See e.~g. \cite{michaeli}
for a thorough approach to ordinary (ungraded) Lie coalgebras.

\section{The Chen perturbation}\label{chenpert}

Let $X$ be a (connected) smooth manifold, and let $\mathcal A(X)$ be its 
ordinary de Rham algebra.
Our convention is that $\mathcal A(X)$ is graded by negative degrees,
that is,
$\mathcal A(X)_j=\mathcal A^{-j}(X)$  for $j \leq 0$.
Let $V$ be a graded vector space
which, to avoid unnecessary complications, we suppose to be of finite type
(that is, finite-dimensional in each degree),
and let $\widehat {\mathrm T}[V]$
denote the graded completion 
of the graded tensor algebra ${\mathrm T}[V]$
relative to the augmentation filtration.
Let 
$\widehat{\mathrm T}_{\mathcal A(X)}[V]$
be the de 
algebra of formal power series in a basis of
$V$ with coefficients in the
de Rham algebra $\mathcal A(X)$ of $X$. 
More formally, let $V^*$ be the graded dual of 
$V$---still of finite type---, 
consider the graded tensor coalgebra $\mathrm T^{\mathrm c}[V^*]$,
and let
\[
\widehat{\mathrm T}_{\mathcal A(X)}[V]=
\mathrm{Hom}(\mathrm T^{\mathrm c}[V^*], \mathcal A(X)).
\]
Thus, the algebra 
$\widehat{\mathrm T}_{\mathcal A(X)}[V]$  
is the appropriate completion 
\[
\mathcal A(X) \widehat{\otimes} {\mathrm T}[V]
\]
of the tensor product
$\mathcal A(X) \otimes {\mathrm T}[V]$.
In his paper \cite{chenfou}, 
Chen referred to $\widehat{\mathrm T}_{\mathcal A(X)}[V]$ 
as the {\em algebra of
$\widehat {\mathrm T}[V]$-valued differential forms on\/} $X$
and defined 
a $\widehat {\mathrm T}[V]$-valued {\em formal power series connection\/}
on $X$ to be a degree $-1$ element of
$\widehat{\mathrm T}_{\mathcal A(X)}[V]$.

Let $\HHHH$ be the de Rham cohomology of $X$, and let
\begin{equation}
   (\HHHH
     \begin{CD}
      \null @>{\nabla}>> \null\\[-3.2ex]
      \null @<<{\pi}< \null
     \end{CD}
    \mathcal A(X), h)
\label{cont0}
  \end{equation}
be a contraction of chain complexes.
Thus, any $\alpha \in \mathcal A(X)$ can be written as
\[
\alpha = (dh + hd + \nabla \pi) \alpha 
=  dh\alpha + \nabla \pi \alpha +hd \alpha
\]
in a unique fashion.
The resulting decomposition
\begin{equation}
\mathcal A(X) = d \mathcal A(X) \oplus \mathrm{ker} (h)  =
d \mathcal A(X) \oplus \mathcal H \oplus h \mathcal A(X) 
\label{hodge}
\end{equation}
where $\mathcal H = \nabla \mathrm H(\mathcal A(X))$
plays the role of the {\em Hodge decomposition\/} in Kodaira-Spencer
deformation theory.
On p.~19 
of \cite{nijritwo},
a decomposition of the kind \eqref{hodge}
(not phrased in the language of contractions)
is indeed referred to as a \lq\lq Hodge decomposition\rq\rq,
and on p.~187 of \cite{chenfou} it is remarked that
a Hodge decomposition of the de Rham complex
of a Riemannian manifold 
is a special case of a decomposition of the kind
\eqref{hodge},  $\mathcal H$ being the space of harmonic forms.

Let $\HH$ be the reduced real {\em homology\/} of $X$,
let $V=s^{-1}\HH$, and consider the algebra
$\widehat {\mathrm T}_{\mathcal A(X)}[s^{-1}\HH]$
of $\widehat {\mathrm T}[s^{-1}\HH]$-valued
differential forms on $X$.
We now recall, in the language of the present paper,
Chen's Theorem 1.3.1 in \cite{chenfou}.

\begin{Theorem}
\label{chen2} The contraction {\rm \eqref{cont0}} determines
a $\widehat {\mathrm T}[s^{-1}\HH]$-valued
formal power series connection
$
\tau \in \widehat {\mathrm T}_{\mathcal A(X)}[s^{-1}\HH]
$
and an algebra  differential $\partial$ on 
$\widehat {\mathrm T}[s^{-1}\HH]$
such that, when $d$ refers to the
ordinary de Rham differential,
relative to the total differential 
\[
d^{\otimes} = d \widehat{\otimes} \mathrm{Id} +  
\mathrm{Id} \widehat{\otimes} \partial
\]
on 
$
\widehat {\mathrm T}_{\mathcal A(X)}[s^{-1}\HH] =
\mathcal A(X) 
 \widehat{\otimes} {\mathrm T}[s^{-1}\HH],
$
the following master equation is satisfied:
\begin{equation}
d^{\otimes} \tau = \tau \tau.
\label{master0}
\end{equation}
\end{Theorem}

We will show below how this theorem is a consequence of a somewhat more 
general result. For intelligibility, at the present stage, we note
the following:
Suppose that $X$ is simply connected.
Then the algebra
${\mathrm T}[s^{-1}\HH]$
is already complete, that is, coincides with
$\widehat {\mathrm T}[s^{-1}\HH]$,
and the algebra
$\widehat {\mathrm T}_{\mathcal A(X)}[s^{-1}\HH]$
comes down to the ordinary tensor product 
$\mathcal A(X) {\otimes} {\mathrm T}[s^{-1}\HH]$.
The differential $\partial$ can be written as an infinite
series
\[
\partial = \partial^1 + \partial^2 + \ldots
\]
where $\partial^1$ is the ordinary cobar construction differential
relative to the coalgebra structure on the real homology 
$\HHHH_*(X)$
of $X$, the
resulting differential graded algebra
$({\mathrm T}[s^{-1}\HH], \partial)$
is then a kind of perturbed reduced cobar construction,
and this perturbed reduced cobar construction is a model
for the real chain algebra of the based loop space of $X$;
we therefore write this differential graded algebra
as $\rcob_{\partial}[\HHHH_*(X)]$.
In particular, when the higher terms $\partial^j$ for $j \geq 2$
are zero, the manifold $X$ is formal over the reals.

Chen already established a version of the
\lq\lq minimality theorem\rq\rq,
though not in the language of $A_{\infty}$-structures.
Indeed, Theorem \ref{chen2},
together with Theorem 3.1.1 in \cite{chen},
include the statement 
that, for simply connected $X$, the real homology $\HHHH_*(X)$
of the manifold
$X$ acquires an $A_{\infty}$-coalgebra structure such that
the chains on $X$ and $\HHHH_*(X)$, endowed with the $A_{\infty}$-coalgebra
structure, are equivalent. 
Dualized, this is precisely the statement of
the
\lq\lq minimality theorem\rq\rq\ over the reals.

\begin{Remark}\label{chenmini} 
Suitably interpreted, Chen's construction makes perfect sense over the rationals.
Let $X$ be a simply connected space
and take $\mathcal A(X)$ to be the graded commutative algebra
of rational forms on $X$.
We continue to denote the resulting
differential graded algebra, now over the rationals, $\mathbb Q$,
by $\rcob_{\partial}[\HHHH_*(X)]$.
The rational homology $\HHHH_*(X)$ acquires a graded cocommutative
coalgebra structure whence the shuffle diagonal map
turns $\rcob_{\partial}[\HHHH_*(X)]$ into a
graded cocommutative Hopf algebra. Thus, by the Milnor-Moore theorem,
$\rcob_{\partial}[\HHHH_*(X)]$ is the universal enveloping algebra
$\mathrm U[L]$ of a differential graded Lie algebra $L$.
This differential graded Lie algebra $L$ is in fact the familiar {\em minimal\/}
Lie algebra model for the based loop space on $X$
nowadays widely used in rational homotopy theory.
This has been worked out in
{\rm \cite{hainone}, \cite{hainthr}, \cite{tanreboo}};
indeed, in the latter reference,
the differential graded Lie algebra $L$, together with
the formal power series connnection,
is referred to as the {\em Chen model\/} for (the rational homotopy type of) the
space $X$. It is in this sense
that the usage of the term \lq\lq minimal\rq\rq\ 
in the present paper is compatible with its usage in rational homotopy theory.

The notion of $C_{\infty}$-coalgebra is lurking behind these
observations.
We shall elucidate this kind of structure in Remark {\rm \ref{cinfco}} below.

\end{Remark}

\section{Perturbations for coalgebras}

Abstracting from the perturbation argument for Chen's theorem
quoted above, V. Gugenheim developed a general perturbation
theory for differential graded coalgebras which includes
and explains Chen's theorem \cite{gugenhth}.
Gugenheim made it entirely clear that his perturbation argument
is formally exactly the same as that of Chen, but placed in a much 
more general context.
We will now explain a version of Gugenheim's
result, somewhat more general than the original one in
\cite{gugenhth}.
The proofs of all the claims in this section 
will be given in
Section \ref{algebra} below.

Let $C$ be a coaugmented differential graded coalgebra and let
\begin{equation}
   (\HHH
     \begin{CD}
      \null @>{\nabla}>> \null\\[-3.2ex]
      \null @<<{\pi}< \null
     \end{CD}
    C, h)
\label{contco}
  \end{equation}
be a contraction of chain complexes.
The situation considered by Gugenheim in \cite{gugenhth} is the special case
where $M$ has zero differential, so that $M$ then amounts to the 
homology of $C$.
In the general case,
the counit $\varepsilon \colon C \to R$ and coaugmentation
$\eta \colon R \to C$ induce a
\lq\lq counit\rq\rq\  $\varepsilon \colon \HHH \to R$ 
and \lq\lq coaugmentation\rq\rq\ 
$\eta \colon R \to \HHH$
for $\HHH$ in such a way that \eqref{contco} is a contraction of
augmented and coaugmented chain complexes.
Thus $\HHH$ admits a direct sum decomposition
$\HHH = R \oplus \JJJ\HHH$;
indeed we can view $\HHH$ as a differential graded coalgebra
with zero diagonal on $\JJJ\HHH$,
and $\JJJ\HHH$ can then be interpreted as 
the coaugmentation coideal of $\HHH$.
The differential graded {\em tensor\/} algebra
$\mathrm T[s^{-1}\JJJ\HHH]$
is then the corresponding reduced {\em cobar\/} construction for $\HHH$;
let $\tau_{\HHH}\colon \HHH \to \mathrm T[s^{-1}\JJJ\HHH]$
be the universal cobar construction
twisting cochain, that is, the desuspension mapping, followed by the
canonical injection.

Let
\begin{equation}
\tau^1 = \tau_{\HHH}\pi\colon C \to s^{-1}\JJJ\HHH\subseteq
\mathrm T[s^{-1}\JJJ\HHH]
\label{proc001}
\end{equation}
and, for $\jj \geq 2$, 
let
\[
\tau^j \colon  C \longrightarrow (s^{-1}\JJJ\HHH)^{\otimes j}
\subseteq \mathrm T[s^{-1}\JJJ\HHH]
\]
be the degree $-1$ morphism defined recursively by
\begin{equation}
\tau^{\jj} = (\tau^1\cup\tau^{{\jj}-1} +  \dots +
\tau^{{\jj}-1}\cup\tau^1)h\colon C \to (s^{-1}\JJJ\HHH)^{\otimes j} .
\label{proc01}
\end{equation}
Thereafter, for $\jj \geq 1$,
define the degree $-1$ derivation
$\mathcal D^j$ on $\mathrm T[s^{-1}\JJJ\HHH]$ 
by
\begin{equation}
 \mathcal D^{\jj}\tau_{\HHH} = (\tau^1 \cup\tau^{\jj} +
\dots + \tau^{\jj} \cup\tau^1)\nabla \colon \HHH \to
 (s^{-1}\JJJ\HHH)^{\otimes (j+1)}
\subseteq \mathrm T[s^{-1}\JJJ\HHH]. \label{proc31}
\end{equation}

Let $\widehat {\mathrm T}[s^{-1}\JJJ\HHH]$
denote the graded completion 
of the graded tensor algebra ${\mathrm T}[s^{-1}\JJJ\HHH]$,
the term completion being taken
relative to the augmentation filtration.

\begin{Theorem}\label{c}
The infinite series
\begin{equation}
\mathcal D = \mathcal D^1 + \mathcal D^2 + \dots \colon
\widehat {\mathrm T}[s^{-1}\JJJ\HHH] \to 
\widehat {\mathrm T} [s^{-1}\JJJ\HHH] \label{proc211}
\end{equation}
is an algebra perturbation of the differential $d$ on
$\widehat {\mathrm T}[s^{-1}\JJJ\HHH]$
induced from the differential on $\HHH$, and the infinite series
\begin{equation}
\tau = \tau^1 + \tau^2 + \dots \colon 
C \to \widehat {\mathrm T}[s^{-1}\JJJ\HHH]\label{proc111}
\end{equation}
is a twisting cochain
\[
C \longrightarrow (\widehat {\mathrm T}[s^{-1}\JJJ\HHH], d + \mathcal D) .
\]
\end{Theorem}

For the special case where $\HHH$ has zero differential
so that $\HHH$ then coincides with the homology $\HHHH(C)$ of $C$, this 
theorem is
essentially Proposition 2.1 in \cite{gugenhth}.
We will henceforth write
\begin{equation}
\widehat {\rcob}_{\mathcal D}\HHH=
(\widehat {\mathrm T}[s^{-1}\JJJ\HHH], d + \mathcal D) .
\label{shcc}
\end{equation}

\smallskip
\noindent{\textbf {Complement 1.}}
{\sl Suppose that, in addition, $\HHH$ is a coaugmented differential graded coalgebra
and that $\nabla$ is a morphism of differential graded coalgebras.
Then $\mathcal D^j$ is zero for $j \geq 2$ and
$\mathcal D^1$ is the algebra perturbation
determined by the coalgebra structure of $\HHH$.
}
\smallskip

Under suitable circumstances,
the graded tensor algebra
$\mathrm T[s^{-1}\JJJ\HHH]$ is already complete.
This happens to be the case when $\HHH$ is {\em simply connected\/}
in the sense that it is concentrated in non-negative
degrees and that $\mathrm J\HHH_1$ is zero,
or when $\mathrm J\HHH$ is concentrated in non-positive degrees.
In this case,
the series
\begin{align}
\tau &= \tau^1 + \tau^2 + \dots \colon 
C \to\mathrm T[s^{-1}\JJJ\HHH],\label{proc11}
\\
\mathcal D &= \mathcal D^1 + \mathcal D^2 + \dots \colon
\mathrm T[s^{-1}\JJJ\HHH] \to \mathrm T[s^{-1}\JJJ\HHH] \label{proc21}
\end{align}
converge naively in the sense that,
applied to a specific element,
only finitely many terms are non-zero.
Furthermore,
$\mathcal D$ is then an algebra perturbation 
of the differential on $\mathrm T[s^{-1}\JJJ\HHH]$,
and
\[
\tau = \tau^1 + \tau^2 + \dots 
\colon C \longrightarrow \rcob_{\mathcal D} \HHH 
\]
is a {\em twisting cochain \/} where we use the notation
\begin{equation}
\rcob_{\mathcal D}\HHH=
({\mathrm T}[s^{-1}\JJJ\HHH], d + \mathcal D) .
\label{shc}
\end{equation}
At the risk of being, perhaps, repetitive, we point out explicitly that
the piece of structure
$\mathcal D$ in $\rcob_{\mathcal D}\HHH$ is precisely an
$A_{\infty}$-coalgebra structure on $\HHH$.

\smallskip
\noindent{\textbf {Complement 2 to Theorem \ref{c}.}}
{\sl
When $C$ is simply connected or when $C$ is concentrated
in non-positive degrees,
the adjoint 
\[
\overline \tau \colon \rcob C \longrightarrow 
\rcob_{\mathcal D}\HHH
\]
of the twisting cochain $\tau$, 
necessarily a morphism of differential graded algebras, 
is a chain equivalence.
If, furthermore, $\HHH$ is a coaugmented differential graded coalgebra
and $\nabla$ is a morphism of differential graded coalgebras,
$\rcob_{\mathcal D}\HHH$
is the ordinary reduced cobar construction on $\HHH$.}

\smallskip
Indeed, in the situation of the \lq\lq Furthermore\rq\rq\ statement
of Complement 2, 
 the 
vanishing of the higher terms  $\mathcal D^j$ for $j \geq 2$
is a consequence of the annihilation property $ h \nabla = 0$ and
the construction of 
the twisting cochain $\tau$ comes essentially
down to \cite{gugenmun}\ (4.1)$^*$.
A result somewhat weaker than the above Complement 2 is
Theorem 3.2 in \cite{gugenhth} which says that
$\overline \tau$ is a homology isomorphism.

Complement 2 to Theorem \ref{c} includes the statement that,
in the simply connected case,
$\HHH$, endowed with the $A_{\infty}$-coalgebra structure
$\mathcal D$, and $C$, endowed with the 
$A_{\infty}$-coalgebra structure associated with the 
differential graded coalgebra structure,
are, via $\tau$, equivalent as $A_{\infty}$-coalgebras.

\begin{Remark} [Lemma 2.2.1 in \cite{gugenhth}] 
Suppose that the differential of $\HHH$ is zero.
Then $\HHH$ amounts to the homology $\HHHH(C)$ of $C$
and acquires the structure of a coaugmented graded coalgebra.
In this case, even though neither the morphism $\nabla$ nor the morphism
$\pi$ in the contraction \eqref{contco} are supposed to be compatible with
the coalgebra structures,
the operator $\mathcal D^1$ is the ordinary cobar construction
differential on the graded tensor algebra ${\mathrm T}[s^{-1}\JJJ\HHHH(C)]$
determined by the diagonal map $\Delta_{\HHHH(C)}$ of $\HHHH(C)$. 
Indeed, write the diagonal map of $C$ as $\Delta$ as usual. 
By construction, the 
diagonal map $\Delta_{\HHH}$ of $\HHH$ coincides with the composite
\begin{equation*}
\begin{CD}
\HHH @>\nabla>> C @>{\Delta}>> C \otimes C @>{\pi \otimes \pi}>>
\HHH \otimes \HHH .
\end{CD}
\end{equation*}
Consequently
\begin{align*}
 \mathcal D^1\tau_{\HHH} &= (\tau^1 \cup\tau^1)\nabla
\\
 &= ( (\tau_{\HHH}\pi) \cup  (\tau_{\HHH}\pi))\nabla
\\
 &=  (\tau_{\HHH} \otimes  \tau_{\HHH})(\pi\otimes \pi)\Delta\nabla
\\
&=  (\tau_{\HHH} \otimes  \tau_{\HHH})\Delta_{\HHH}.
\end{align*}
However the identity
\[
\mathcal D^1\tau_{\HHH} =
(\tau_{\HHH} \otimes  \tau_{\HHH})\Delta_{\HHH}
\]
says that $\mathcal D^1$ is the ordinary cobar construction
differential.
Thus, in this case,
$\rcob_{\mathcal D}\HHHH(C)$ is a perturbation of the
ordinary reduced cobar construction $\rcob\HHHH(C)$
over $\HHHH(C)$ and thence endows $\HHHH(C)$ with an
$A_{\infty}$-coalgebra structure equivalent to $C$.

\end{Remark}

\begin{Remark}\label{cinfco}
The concept dual to that of a $C_{\infty}$-algebra is that of a
$C_{\infty}$-coalgebra: Let $C$ be a coaugmented differential graded
$A_{\infty}$-coalgebra, with standard construction
$\rcob_{\partial}(C)$. Then this $A_{\infty}$-coalgebra structure on $C$
is a $C_{\infty}$-coalgebra structure provided the shuffle diagonal
turns $\rcob_{\partial}(C)$ into a differential graded Hopf algebra,
necessarily graded cocommutative. 
Exactly the same reasoning as that in Remark {\rm \ref{chenmini}} reveals that
the standard construction $\rcob_{\partial}(C)$ 
of a coaugmented $C_{\infty}$-coalgebra $C$ 
is the universal enveloping algebra
$\mathrm U[L]$ of a differential graded Lie algebra $L$ which, in turn, is a perturbation
of the free differential graded Lie algebra generated by $s^{-1}(\mathrm J C)$.
Thus the standard construction of a (coaugmented) $C_{\infty}$-coalgebra
comes down to a {\em perturbation of the
free differential Lie algebra generated by\/} $s^{-1}(\mathrm J C)$.

In the special case where, as a differential graded coalgebra, $C$ is an ordinary
(coaugmented) cocommutative differential graded coalgebra,
the shuffle diagonal plainly turns  the ordinary cobar construction $\rcob(C)$
into a differential graded Hopf algebra;
this situation
has been explored 
by J. Moore in {\rm \cite{mooresix}} and  {\rm \cite{moorefiv}}. 

\end{Remark}

\section{Perturbations for algebras}

We now spell out the situation dual to that in the previous
section. Again the proofs of all the claims in this section 
will be given in
Section \ref{algebra} below.

Thus,
let $A$ be an augmented differential graded algebra and let
\begin{equation}
   (\HHH
     \begin{CD}
      \null @>{\nabla}>> \null\\[-3.2ex]
      \null @<<{\pi}< \null
     \end{CD}
    A, h)
\label{contalg}
  \end{equation}
be a contraction of {\em chain complexes\/}.
The unit $\eta \colon R \to A$ and augmentation
$\varepsilon \colon A \to R$ induce a
\lq\lq unit\rq\rq\  $\eta \colon R \to \HHH$ 
and \lq\lq augmentation\rq\rq\ 
$\varepsilon \colon \HHH \to R$
for $\HHH$ in such a way that \eqref{contalg} is a contraction of
augmented and coaugmented chain complexes.
Thus $\HHH$ admits a direct sum decomposition
$\HHH = R \oplus \III\HHH$;
indeed we can view $\HHH$ as a differential graded algebra
with zero multiplication on $\III\HHH$,
and $\III\HHH$ can then be interpreted as 
the augmentation ideal of $\HHH$.
The differential graded {\em tensor\/} coalgebra
$\mathrm T^{\mathrm c}[s\III\HHH]$
is then the corresponding reduced {\em bar\/} construction for $\HHH$;
let $\tau_{\HHH}\colon \mathrm T^{\mathrm c}[s\III\HHH]\to \HHH$
be the universal bar construction
twisting cochain, that is, the 
canonical projection to $s\III\HHH$,
followed by the desuspension mapping.

Let
\begin{equation}
\tau^1 = \nabla \tau_{\HHH}\colon  \mathrm T^{\mathrm c}[s\III\HHH]
\to  s\III\HHH\to A
\label{proc0001}
\end{equation}
and, for $\jj \geq 2$, 
let
\[
\tau^j \colon  
\mathrm T^{\mathrm c}[s\III\HHH]
\to  (s\III\HHH)^{\otimes j}\to A
\]
be the degree $-1$ morphism defined recursively by
\begin{equation}
\tau^{\jj} = h(\tau^1\cup\tau^{{\jj}-1} +  \dots +
\tau^{{\jj}-1}\cup\tau^1)\colon (s\III\HHH)^{\otimes j}\to A.
\label{proc011}
\end{equation}
Thereafter, for $\jj \geq 1$,
define the degree $-1$ coderivation
$\mathcal D^j$ on 
$\mathrm T^{\mathrm c}[s\III\HHH]$
by
\begin{equation}
\tau_{\HHH} \mathcal D^{\jj} =  \pi (\tau^1\cup \tau^{\jj} +
\dots + \tau^{\jj}\cup \tau^1) 
\colon \mathrm T^{\mathrm c}[s \mathrm I\HHH] _{\jj+ 1}^{\mathrm c}
= (s\III\HHH)^{\otimes (j+1)}
 \to
\HHH. \label{proc333}
\end{equation}
In particular, for $\jj \geq 1$, the coderivation $\mathcal
D^{\jj}$ is zero on $\mathrm F_j \mathrm T^{\mathrm c}[s \mathrm I\HHH] $ 
and lowers
coaugmentation filtration by $\jj$.

\begin{Theorem}\label{alg}
The infinite series
\begin{equation}
\mathcal D = \mathcal D^1 + \mathcal D^2 + \dots \colon
\mathrm T^{\mathrm c}[s\III\HHH] \to \mathrm T^{\mathrm c}[s\III\HHH]
 \label{proc2111}
\end{equation}
is a coalgebra perturbation of the differential $d$ on
$\mathrm T^{\mathrm c}[s\III\HHH]$
induced from the differential on $\HHH$, and the infinite series
\begin{equation}
\tau = \tau^1 + \tau^2 + \dots \colon 
\mathrm T^{\mathrm c}[s\III\HHH] \to A
\label{proc1111}
\end{equation}
is a twisting cochain
\[
(\mathrm T^{\mathrm c}[s\III\HHH], d + \mathcal D)
\longrightarrow A.
\]
Furthermore, the adjoint 
\[
\overline \tau \colon \rbar_{\mathcal D}\HHH
\longrightarrow 
\rbar A
\]
of the twisting cochain $\tau$, necessarily 
a morphism of differential graded coalgebras, is a chain equivalence.
\end{Theorem}

To our knowledge, such a result for a general chain complex $\HHH$ 
appears in the literature for the first time
in \cite{gugensta}---the special case where the ground ring is a field
and where $\HHH$ has zero homology is due to Chen, as noted earlier.
The sums \eqref{proc2111} and \eqref{proc1111} are in general infinite.
However, applied to a specific element which, since
$\mathrm T^{\mathrm c}[s\III\HHH]$ 
is cocomplete, necessarily lies in some finite
filtration degree subspace, since the operators $\mathcal D^{\jj}$
($\jj \geq 1$) lower coaugmentation filtration by $\jj$, only
finitely many terms will be non-zero, whence the convergence is
naive.
We will henceforth write
\begin{equation}
\rbar_{\mathcal D}\HHH=
(\mathrm T^{\mathrm c}[s\III\HHH], d + \mathcal D) .
\label{sha}
\end{equation}
The piece of structure
$\mathcal D$ in
$\rbar_{\mathcal D}\HHH$ is precisely an
$A_{\infty}$-algebra structure on $\HHH$
and Theorem \ref{alg} includes the statement that
$\HHH$, endowed with the $A_{\infty}$-algebra structure
$\mathcal D$, and $A$, endowed with the 
$A_{\infty}$-algebra structure associated with the 
differential graded algebra structure,
are, via $\tau$, equivalent as $A_{\infty}$-algebras.
This recovers, in particular,
Kadeishvili's result, Theorem \ref{kadea}.

\smallskip
\noindent{\textbf {Complement to Theorem \ref{alg}.}}
{\sl Suppose that, in addition, $\HHH$ is an augmented differential graded algebra
and that $\pi$ is a morphism of differential graded algebras.
Then $\mathcal D^j$ is zero for $j \geq 2$, the operator
$\mathcal D^1$ is the coalgebra perturbation
determined by the algebra structure of $\HHH$,
and $\rbar_{\mathcal D}\HHH$ is the ordinary reduced bar construction for
$\HHH$.}

\smallskip
Indeed, in the situation of the Complement, the 
vanishing of the higher terms  $\mathcal D^j$ for $j \geq 2$
is a consequence of the annihilation property $\pi h = 0$ and the
construction of the 
twisting cochain comes essentially
down to \cite{gugenmun}\ (4.1)$_*$.

In the special case where $A$ has zero differential
and the original contraction \eqref{contalg}
is the trivial contraction of the kind
\begin{equation}
   (A
     \begin{CD}
      \null @>{\mathrm{Id}}>> \null\\[-3.2ex]
      \null @<<{\mathrm{Id}}< \null
     \end{CD}
    A, 0) ,
\label{trivial}
  \end{equation}
$\HHH$ and $A$ coincide, the perturbation
$\ppartial$ coincides with the perturbation $\partial$ determined
by the algebra structure on $A$, and
$\mathrm T_{\ppartial}^{\mathrm c}[s\HHH]$
coincides with the ordinary reduced
bar construction for $A$; the  twisting cochain $\tau$ then comes
down to the bar construction twisting cochain
and in fact coincides with $\tau^1$.
In this case, the higher terms $\tau^{\jj}$ and
$\ppartial^{\jj}$ ($\jj \geq 2$) are obviously zero, and the
operator $\ppartial^1$ manifestly coincides with the 
bar construction operator.

Likewise, suppose that the differential of $\HHH$ is zero.
Then $\HHH$ amounts to the homology $\HHHH(A)$ of $A$
and acquires the structure of an augmented graded algebra.
In this case, even though neither the morphism $\nabla$ nor the morphism
$\pi$ in the contraction \eqref{contalg} are supposed to be compatible with
the algebra structures,
the operator $\mathcal D^1$ is the ordinary bar construction
differential on the graded tensor coalgebra 
$\mathrm T^{\mathrm c}[s\III \HHHH(A) ]$ 
determined by the multiplication map of $\HHHH(A)$,
and $\rbar_{\mathcal D}\HHHH(A)$ is a perturbation of the ordinary reduced
bar construction $\rbar\HHHH(A)$.

\begin{Remark}\label{kami1}{\rm [Relationship with Kadeishvili's minimality 
theorem for algebras]}
Comparison of the proof of Theorem {\rm \ref{kadea}}
with that of Theorem {\rm \ref{alg}} reveals the following:
The terms $\tau^j$ in
the proof of Theorem {\rm \ref{alg}} exhibit precisely terms of the kind
$f_j$ in the proof of Theorem {\rm \ref{kadea}},
and the operators $\mathcal D^j$ 
in the proof of Theorem {\rm \ref{alg}} yield terms of the kind $m_j$
in the proof of Theorem {\rm \ref{kadea}}.
The coalgebra structure of 
$\mathrm T^{\mathrm c}[s\HHH]$ exploited in the proof of
 Theorem {\rm \ref{alg}}
organizes the otherwise tedious calculations in the original proof
of Theorem {\rm \ref{kadea}}, and the usage,
in
the proof of Theorem {\rm \ref{alg}},
of the chain homotopy $h$
in the contraction \eqref{contalg} removes the ambiguities
related with the choices of the bounding chains
in the proof of Theorem {\rm \ref{kadea}}
and thus leads to an algorithm, cf. {\rm \cite{vejjoh}}.

\end{Remark}

\section{A proof of Chen's theorem}

We will now briefly explain how Theorem \ref{alg} includes
Theorem \ref{chen2}; this will make it clear that
the basic perturbation argument goes back to Chen: 
Given the smooth manifold $X$, let
$A=\mathcal A(X)$ 
and pick a contraction of the kind {\rm \eqref{cont0}}.
Thus $\HHHH$ is then the de Rham cohomology algebra of $X$.
The recursive construction \eqref{proc011}
yields the twisting cochain
\[
\tau = \tau^1 + \tau^2 + \dots \colon 
\rbar_{\mathcal D}\HHHH \longrightarrow \mathcal A(X)
\]
and the formulas \eqref{proc333}
yield the coalgebra differential
\[
\mathcal D = \mathcal D^1 + \mathcal D^2 + \dots \colon
\mathrm T^{\mathrm c}[s\III\HHHH] \to \mathrm T^{\mathrm c}[s\III\HHHH].
\]
To simplify the exposition we will suppose that
the homology of $X$ is of finite type.
Then the differential graded algebra which is the real
dual of $\rbar_{\mathcal D}\HHHH$
is precisely a differential graded algebra of the kind
\[
\widehat{\rcob}_{\partial}[\HHHH_*(X)]
\]
where $\partial$ is the algebra differential dual to the coalgebra
differential $\mathcal D$.
Thus the twisting cochain $\tau$ then appears as an element of the
differential graded algebra
\begin{equation*}
\mathrm{Hom}(\rbar_{\mathcal D}\HHHH,\mathcal A(X))\cong
\widehat {\mathrm T}_{\mathcal A(X)}[s^{-1}\HH] =
\mathcal A(X) 
\widehat{\otimes} {\mathrm T}[s^{-1}\HH],
\end{equation*}
endowed with the total differential
\[
d^{\otimes} = d \widehat{\otimes} \mathrm{Id} +  
\mathrm{Id} \widehat{\otimes} \partial.
\]
The twisting cochain property says that $\tau$
satisfies the  master equation \eqref{master0}.
The formulas  \eqref{proc011}
and  \eqref{proc333} are then essentially the same as those 
used by Chen to establish the existence of the formal power series
connection and of the differential $\partial$ in the proof of his Theorem
1.3.1 in \cite{chenfou}.

\section{Homological perturbations and algebraic structure}
\label{algebra}

In the 1980's, I noticed that various standard HPT-constructions
are compatible with algebraic structure,
and I used this observation
to exploit $A_{\infty}$-structures arising in group cohomology
via HPT-constructions of suitable small free resolutions.
These  small free resolutions
enabled me to do 
explicit numerical calculations in group cohomology
which until today are still not doable by other methods.
In particular,
spectral sequences show up which do not collapse from $E_2$.
This illustrates a typical phenomenon:
Whenever a spectral sequence arises from a certain mathematical
structure, a certain strong homotopy structure is lurking behind and
the spectral sequence is an invariant thereof.
The higher homotopy structure is actually finer than the
spectral sequence itself.
The results have been published in the papers 
\cite{perturba}--\cite{holomorp}.

The observation that compatibility with algebraic structure
is hidden in various standard HPT-constructions led to
an alternate approach to Theorem \ref{c} and Theorem \ref{alg}
and, indeed, leads to considerable generalization,
cf. Remark \ref{general} below.
In the academic year 1987/88, lifting of the restrictions
in the USSR enabled T. Kadeishvili to accept an invitation
to the mathematics department of the University of Heidelberg.
During that period,
in collaboration, T. Kadeishvili and I 
developed HPT for general chain equivalences
and, within this research collaboration, 
we worked  out  in particular the
alternate approach.
This kind of approach was also worked out in
\cite{gugenlam}--\cite{gulstatw}.

We will now explain the outcome of this alternate
approach and how it actually leads to
proofs of these theorems and to additional insight.
To this end, 
let 
\begin{equation}
   \left(M
     \begin{CD}
      \null @>{\nabla}>> \null\\[-3.2ex]
     \null @<<{\ppi}< \null
     \end{CD}
   N , h \right)
\label{266}
  \end{equation}
be a filtered contraction. For intelligibility, we recall the following.

\begin{Lemma}[Ordinary perturbation lemma]
\label{ordinary} 
Let $\partial$ be a perturbation of the differential
on $N$, and let
\begin{align} \dell &= \sum_{n\geq 0} \ppi\partial (-h\partial)^n\nabla =
\sum_{n\geq 0} \ppi(-\partial h)^n\partial\nabla
\label{1.1.1}
\\
\nabla_{\partial}&= \sum_{n\geq 0} (-h\partial)^n\nabla
\label{1.1.2}
\\
\ppi_{\partial}&= \sum_{n\geq 0} \ppi(-\partial h)^n
\label{1.1.3}
\\
h_{\partial}&=-\sum_{n\geq 0} (-h\partial)^n h =-\sum_{n\geq 0}
h(-\partial h)^n .
\label{1.1.4}
\end{align}
When the filtrations on $M$ and $N$ are complete, these infinite
series converge, $\dell$ is a perturbation of the differential on
$M$ and, when $N_{\partial}$ and $M_{\dell}$ refer to the new
chain complexes,
\begin{equation}
   \left(M_{\dell}
     \begin{CD}
      \null @>{\nabla_{\partial}}>> \null\\[-3.2ex]
     \null @<<{\ppi_{\partial}}< \null
     \end{CD}
   N_{\partial} , h_{\partial} \right)
\label{2.66}
  \end{equation}
constitute a new filtered contraction that is natural in terms of
the given data.
\end{Lemma}

\begin{proof} See \cite{rbrown} or \cite{gugenhtw}.
\end{proof}

The issue addressed in the perturbation lemma has been
described in the literature as a {\em transference\/} problem.
The ordinary perturbation lemma solves this transference problem
relative merely to a perturbation of the differential on the larger object
$N$.
We will now explore the transference problem 
relative to additional structure.

Let
$\mathrm T[N]$ and $\mathrm T[M]$
be the differential 
graded tensor algebras on 
$N$ and $M$ respectively,
 denote the multiplication map of $\mathrm T[N]$ by $m$,
let $\mathrm T\ppi$ and $ \mathrm T\nabla$ 
be the induced morphisms of differential graded algebras,
and define an operator 
$\mathrm Th \colon \mathrm T[N] \to \mathrm T[N]$
by means of
\[
(\mathrm Th \vert N^{\otimes k}) =
h \otimes (\nabla \ppi)^{\otimes (k-1)}
+
\mathrm{Id} \otimes h \otimes (\nabla \ppi)^{\otimes (k-2)}
+
\dots
+
\mathrm{Id}^{\otimes (k-1)} \otimes h,\quad k \geq 1.
\]
This is an instance of what is referred to as the {\em tensor trick\/},
developed in \cite{homotype}
and exploited in \cite{gulstatw}, \cite{homotype},  
\cite{perturba}, \cite{cohomolo}, \cite{huebkade}
and elsewhere; cf. also \S 3 of \cite{gulasta}.
We will come back to the tensor trick in Section \ref{labell}
below.

With the above preparations out of the way, the morphisms $m$,
 $\mathrm T\ppi$, $\mathrm Th$, etc. are related by the identities
\[
m (\mathrm{Id} \otimes \mathrm Th + 
\mathrm Th \otimes (\mathrm T\nabla)(\mathrm T\ppi)) = (\mathrm Th)\,m,
\]
and 
\[
D(\mathrm Th) (=d(\mathrm Th) + (\mathrm Th)d) 
= (\mathrm T\nabla)(\mathrm T\ppi) - \mathrm{Id},
\]
that is,
$\mathrm Th$ is a homotopy 
$\mathrm{Id}\simeq (\mathrm T\nabla)(\mathrm T\ppi) $
of morphisms of differential graded algebras,
whence the data
\begin{equation}
   \left(\mathrm T[M]
     \begin{CD}
      \null @>{\mathrm T\nabla}>> \null\\[-3.2ex]
     \null @<<{\mathrm T\ppi}< \null
     \end{CD}
   \mathrm T[N] ,\mathrm T h \right)
\label{2.2.0}
\end{equation}
constitute a contraction of augmented algebras.
Further, with respect to the augmentation filtrations
on $\mathrm T[N]$ and $\mathrm T[M]$,
these data constitute in fact 
a contraction of filtered algebras.
This idea  goes back to Theorem 12.1 in 
{\sc Eilenberg and Mac Lane I}~\cite{eilmactw}
where it is spelled out in the dual situation
as a {\em contraction of bar constructions\/}.
The  contraction  of coalgebras that corresponds to \eqref{2.2.0}
is also spelled out  in 
\cite{gugenlam} (3.2),
in 
\cite{gulasta} (\S 3), and in \cite{gulstatw}  (2.2).

Given a chain complex $X$ and a multiplicative perturbation
$\partial$ of the algebra differential
on $\mathrm T[X]$, we shall denote the new differential graded
algebra by $\mathrm T_{\partial}[X]$.
Maintaining terminology introduced in \cite{huebkade},
we shall refer to
a chain complex $X$ 
(without any additional structure)
as being {\em connected in the reduced sense\/}
provided it
is zero in degree zero and, furthermore, either non-negative
or non-positive;
the reader will note that 
\lq\lq connected in the reduced sense\rq\rq\ 
does {\em not\/} coincide with the standard usage of the
term \lq\lq connected\rq\rq.
On the other hand, an augmented differential graded algebra
is  {\em connected\/} in the usual sense
if and only  its augmentation ideal is,  as a chain complex,
connected in the reduced sense.

Henceforth a tensor algebra $\mathrm T[W]$
on a graded $R$-module $W$ will always be viewed as a
{\em filtered algebra\/}
with respect to the {\em augmentation filtration\/}.
Here is Theorem 2.2$^*$ of \cite{huebkade}.

\begin{Theorem} \label{2.2a}
Suppose that $\mathrm T[N]$ and $\mathrm T[M]$ are
connected and let 
$\partial$ be
a multiplicative perturbation of the differential on $\mathrm T[N]$
with respect to the augmentation filtration.
Then the perturbation 
$
\mathcal D
$
given by {\rm \eqref{1.1.1}}
together with the morphisms
given by {\rm \eqref{1.1.2} -- \eqref{1.1.4}}
yield
a contraction
\begin{equation}
   \left(\mathrm T_{\mathcal D}[M]
     \begin{CD}
      \null @>{\mathrm T_{\partial}\nabla}>> \null\\[-3.2ex]
     \null @<<{\mathrm T_{\partial}\ppi}< \null
     \end{CD}
   \mathrm T_{\partial}[N] ,\mathrm T_{\partial} h \right)
\label{2.2.1}
\end{equation}
of filtered 
differential 
graded
algebras which is natural in terms of the given data.
\end{Theorem}

Plainly the perturbation 
$
\mathcal D
$
on $\mathrm T[M]$ and the
morphisms 
$\mathrm T_{\partial}\nabla,\,
\mathrm T_{\partial}\ppi,\,
\mathrm T_{\partial}h $
are determined by their restrictions to
$N \subseteq \mathrm T[N]$
and
$M \subseteq \mathrm T[M]$
as appropriate.

The proof of Theorem \ref{2.2a} given in \cite{huebkade} relies on
the Algebra Perturbation Lemma
(\cite{huebkade} Lemma 2.1$^*$)
and involves 
the \lq\lq tensor trick\rq\rq.

We will now sketch a proof of Theorem \ref{c}.
Let $C$ be a coaugmented
differential graded coalgebra, 
with structure maps $\Delta$ and $\eta$,
and let
\[
\Omega C
=
\mathrm T_{\partial}[s^{-1}(\JJJ C)]
\]
be its cobar construction,
$\partial$ being the derivation
on the 
differential graded tensor algebra
$\mathrm T[s^{-1}(\JJJ C)]$ induced by the diagonal map of $C$.
Thus
with respect to the augmentation filtration,
$\mathrm T[s^{-1}(\JJJ C)]$
is a filtered differential graded algebra,
$\partial$ is a multiplicative perturbation, and
the cobar construction $\Omega C$ appears as a
\lq\lq perturbation\rq\rq\ 
of
$\mathrm T[s^{-1}(\JJJ C)]$.

Suppose momentarily that 
$C$ is  simply connected or concentrated in non-positive degrees.
Then  $N= s^{-1}(\JJJ C)$ is connected in the reduced sense, and
the algebra
$\mathrm T[s^{-1}(\JJJ C)]$
is complete.
The contraction \eqref{contco}
determines
a contraction
\Nsddata N\ppi{\nabla}{s^{-1}\JJJ M}h
of the kind used in  Theorem \ref{2.2a},
where now $s^{-1}\JJJ M$ plays the role of $M$ in  Theorem \ref{2.2a}.
Theorem \ref{2.2a} then yields a multiplicative perturbation
$\mathcal D$ 
of the differential
on 
$\mathrm T[{s^{-1}\JJJ M}]$
and a contraction
\begin{equation}
   \left(\mathrm T_{\mathcal D}[{s^{-1}\JJJ M}]
     \begin{CD}
      \null @>{\mathrm T_{\partial}\nabla}>> \null\\[-3.2ex]
     \null @<<{\mathrm T_{\partial}\ppi}< \null
     \end{CD}
   \Omega C ,\mathrm T_{\partial} h \right)
\label{2.2.12}
\end{equation}
of filtered differential graded algebras.
The naturality of the constructions implies that
$\mathrm T_{\partial}\ppi$
is the adjoint of a twisting cochain
\[
\tau\colon C \longrightarrow \mathrm T_{\mathcal D}[{s^{-1}\JJJ M}].
\]
A closer look reveals that
$\tau$ and $\mathcal D$ actually coincide with
\eqref{proc11} and \eqref{proc21}, respectively.
This establishes the statement of Complement 2 to Theorem \ref{c}
and yields, furthermore, explicit morphisms
which then can be extended, by suitable HPT-constructions,
to $A_{\infty}$-morphisms between $C$ and $M$ (with its
$A_{\infty}$-coalgebra structure)
and thus establish 
an $A_{\infty}$-equivalence  between $C$ and $M$.

To establish Theorem \ref {c} for a general coaugmented differential 
graded coalgebra $C$, we note that the contraction 
\eqref{2.2.0} of filtered algebras induces a contraction
\begin{equation}
   \left(\widehat{\mathrm T}[M]
     \begin{CD}
      \null @>{\widehat{\mathrm T}\nabla}>> \null\\[-3.2ex]
     \null @<<{\widehat{\mathrm T}\ppi}< \null
     \end{CD}
   \widehat{\mathrm T}[N] ,\widehat{\mathrm T} h \right)
\label{2.2.00}
\end{equation}
of complete algebras.
The statement of Theorem \ref{2.2a} extends to that situation
and a proof of Theorem \ref {c} for a general coaugmented differential 
graded coalgebra $C$ can then be concocted, just as for the
particular case handled first.

Finally we will indicate the
necessary modifications
to arrive at a proof of Theorem \ref{alg}.
Instead of the contraction \eqref{2.2.0}, we now consider the
corresponding contraction
\begin{equation}
   \left(\mathrm T^{\mathrm c} [M]
     \begin{CD}
      \null @>{\mathrm T^{\mathrm c}\nabla}>> \null\\[-3.2ex]
     \null @<<{\mathrm T^{\mathrm c}\ppi}< \null
     \end{CD}
   \mathrm T^{\mathrm c}[N] ,\mathrm T^{\mathrm c} h \right)
\label{2.2.0c}
\end{equation}
of coaugmented coalgebras.
Relative to the coaugmentation filtrations,
the  coalgebra version of Theorem \ref{2.2a}
is true without any connectivity assumption and takes the following form;
actually this is Theorem 2.2$_*$ of \cite{huebkade}, not spelled
out explicitly there.

\begin{Theorem} \label{2.2c}
Let
$\partial$ be
a coalgebra perturbation of the differential on $\mathrm T^{\mathrm c}[N]$
with respect to the coaugmentation filtration.
Then the perturbation 
$
\mathcal D
$
given by {\rm \eqref{1.1.1}}
together with the morphisms
given by {\rm \eqref{1.1.2} -- \eqref{1.1.4}}
yield
a contraction
\begin{equation}
   \left(\mathrm T^{\mathrm c}_{\mathcal D}[M]
     \begin{CD}
      \null @>{\mathrm T^{\mathrm c}_{\partial}\nabla}>> \null\\[-3.2ex]
     \null @<<{\mathrm T^{\mathrm c}_{\partial}\ppi}< \null
     \end{CD}
   \mathrm T^{\mathrm c}_{\partial}[N] ,\mathrm T^{\mathrm c}_{\partial} h \right)
\label{2.2.1.c}
\end{equation}
of filtered 
differential 
graded
coalgebras which is natural in terms of the given data.
\end{Theorem}

The proof of Theorem \ref{2.2c} relies on
the Coalgebra Perturbation Lemma
(\cite{huebkade} Lemma 2.1$_*$)
and involves  likewise
the \lq\lq tensor trick\rq\rq.

Dualizing the above reasoning which leads to a proof of
Theorem \ref{c}, the reader is now invited to 
concoct a proof of  Theorem \ref{alg}.
We refrain from spelling out details.

\section{General $A_{\infty}$-algebras and $A_{\infty}$-coalgebras }
\label{ainf}

The reasoning in the previous section extends immediately to general
$A_{\infty}$-algebras and $A_{\infty}$-coalgebras
and thus yields solutions of the corresponding transference problems:

Let $A$ be an augmented $A_{\infty}$-algebra,
the  $A_{\infty}$-algebra structure being given 
by a coalgebra perturbation $\partial$ 
of the differential on $\mathrm T^{\mathrm c}[s \III A]$
relative to the coaugmentation filtration and, as before,
write $\rbar_{\partial }A$ for the perturbed
differential graded coalgebra.
Moreover, let
\begin{equation}
   (s\III\HHH
     \begin{CD}
      \null @>{\nabla}>> \null\\[-3.2ex]
      \null @<<{\pi}< \null
     \end{CD}
    s\III A, h)
\label{contalginf}
  \end{equation}
be a contraction of {\em chain complexes\/}.
Such a contraction arises plainly from a contraction
of augmented chain complexes from $A$ onto $\HHH$.

\begin{Theorem} \label{2.2ainf}
The perturbation 
$
\mathcal D
$
given by {\rm \eqref{1.1.1}}
yields an augmented $A_{\infty}$-algebra structure on $\HHH$,
and the morphisms
given by {\rm \eqref{1.1.2} -- \eqref{1.1.4}}
yield
a contraction
\begin{equation}
   \left( \rbar_{\mathcal D}\HHH
     \begin{CD}
      \null @>{\mathrm T^{\mathrm c}_{\partial}\nabla}>> \null\\[-3.2ex]
     \null @<<{\mathrm T^{\mathrm c}_{\partial}\ppi}< \null
     \end{CD}
\rbar_{\partial }A
,\mathrm T^{\mathrm c}_{\partial} h \right)
\label{2.2.1.cc}
\end{equation}
of filtered differential graded coalgebras which is natural in terms 
of the given data.
\end{Theorem}

\begin{proof} This is an immediate consequence of
Theorem \ref{2.2c}.
\end{proof}

For the special case where $\HHH$ is the homology of $A$, Theorem
\ref {2.2ainf} yields the \lq\lq minimality theorem\rq\rq\ 
for general augmented $A_{\infty}$-algebras.

Likewise let $C$ be a coaugmented $A_{\infty}$-coalgebra
that is  simply connected or concentrated in non-positive degrees,
the  $A_{\infty}$-coalgebra structure being given 
by an algebra perturbation $\partial$ 
of the differential on $\mathrm T[s^{-1}\JJJ C]$
relative to the augmentation filtration and, as before,
write $\rcob_{\partial}C$ for the perturbed
differential graded coalgebra.
Moreover,
let
\begin{equation}
   (s\JJJ\HHH
     \begin{CD}
      \null @>{\nabla}>> \null\\[-3.2ex]
      \null @<<{\pi}< \null
     \end{CD}
    s\JJJ C, h)
\label{contcoalginf}
  \end{equation}
be a contraction of {\em chain complexes\/}.
Such a contraction arises plainly from a contraction
of coaugmented chain complexes from $C$ onto $\HHH$.

\begin{Theorem} \label{2.2cinf}
The perturbation 
$
\mathcal D
$
given by {\rm \eqref{1.1.1}}
yields a coaugmented $A_{\infty}$-coalgebra structure on $\HHH$,
and the morphisms
given by {\rm \eqref{1.1.2} -- \eqref{1.1.4}}
yield
a contraction
\begin{equation}
   \left(\rcob_{\mathcal D}M
     \begin{CD}
      \null @>{\mathrm T_{\partial}\nabla}>> \null\\[-3.2ex]
     \null @<<{\mathrm T_{\partial}\ppi}< \null
     \end{CD}
   \rcob_{\partial}C ,\mathrm T_{\partial} h \right)
\label{2.2.1.aa}
\end{equation}
of filtered 
differential 
graded
algebras which is natural in terms of the given data.
\end{Theorem}

\begin{proof} This is an immediate consequence of
Theorem \ref{2.2a}.
\end{proof}

For the special case where $\HHH$ is the homology of $C$, Theorem
\ref {2.2cinf} yields the \lq\lq minimality theorem\rq\rq\ 
for general coaugmented $A_{\infty}$-coalgebras.

\begin{Remark}\label{general}
At the risk of making a mountain out of a molehill we
point out that
Theorem {\rm \ref{2.2ainf}} is more general than
Theorem  {\rm \ref{alg}} since, in 
Theorem {\rm \ref{2.2ainf}}, $A$ is a general (augmented)
$A_{\infty}$-algebra;
likewise, 
Theorem {\rm \ref{2.2cinf}} is more general than
Theorem  {\rm \ref{c}} since, in
Theorem {\rm \ref{2.2cinf}}, $C$ is a general (coaugmented)
$A_{\infty}$-coalgebra.
In other words, 
Theorem {\rm \ref{2.2ainf}}
provides a solution of the
transference problem for $A_{\infty}$-algebra structures
and
Theorem {\rm \ref{2.2cinf}}
provides a solution of the
transference problem for $A_{\infty}$-coalgebra structures.
\end{Remark}

\section{Summation over oriented rooted planar trees}\label{labell}

In \cite{kontsotw} it has been observed that 
$A_{\infty}$-algebra structures of the kind reproduced in previous
sections can be described in terms of sums over oriented 
rooted planar trees endowed with suitable labels.
Indeed the authors of \cite{kontsotw} bravely acknowledged
that the  oriented rooted planar trees description
is essentially a reworking of the earlier HPT-constructions.
We will now explain
how these sums over oriented rooted 
planar trees come out of the HPT-constructions.

We return to the circumstances of Theorem \ref{alg} above.
We denote the multiplication map
of $A$ by $\mu\colon A \otimes A \to A$.

By construction, the operation $m_2\colon \HHH \otimes \HHH \to \HHH$
is the composite
\begin{equation*}
\begin{CD}
M \otimes M @>{\nabla \otimes \nabla}>> A \otimes A
@>\mu>> A @>\pi>> M .
\end{CD}
\end{equation*}
This is interpreted as an oriented rooted
planar tree with three edges and four vertices,
one vertex where the three edges meet and an end point vertex for each edge.
The vertex where three edges meet is labelled $\mu$, the root vertex
is labelled $\pi$, and the two remaining vertices
are labelled $\nabla$.
The two edges having a vertex labelled $\nabla$ are oriented from
$\nabla$ to $\mu$ whereas the edge having $\mu$ and $\pi$ as vertices
is oriented from $\mu$ to $\pi$.

For $j \geq 1$, simply by construction,
the homogeneous degree $j-1$ operation 
\[
m_{j+1}=s^{-1} \circ \mathcal D^j\circ s^{\otimes (j+1)}
\colon (\III\HHH)^{\otimes (j+1)}
\longrightarrow \III\HHH
\]
is given by
\[
 \pi \circ(\tau^1\cup \tau^{\jj} +
\dots + \tau^{\jj}\cup \tau^1)\circ s^{\otimes (j+1)}
\]
and is thus the sum of the $\jj$ terms
$\pi \circ (\tau^\ell \cup \tau^{\jj+1 -\ell})\circ s^{\otimes (j+1)}$
where $1 \leq \ell \leq \jj$.
Each of these $\jj$ terms can be represented by an
 oriented rooted
planar tree with suitable labels, in the following way:

The operation $m_2$
has already been dealt with.
The operation $m_3\colon \HHH^{\otimes 3} \to \HHH$
is the sum of the two composite morphisms
\begin{equation*}
\begin{CD}
M \otimes M \otimes M @>{\nabla^{\otimes 3}}>> A \otimes A \otimes A
@>{A\otimes \mu}>> A \otimes A
@>{A\otimes h}>> A \otimes A
@>\mu>> A @>\pi>> M 
\end{CD}
\end{equation*}
and
\begin{equation*}
\begin{CD}
M \otimes M \otimes M @>{\nabla^{\otimes 3}}>> A \otimes A \otimes A
@>{\mu\otimes A}>> A \otimes A
@>{h\otimes A}>> A \otimes A
@>\mu>> A @>\pi>> M .
\end{CD}
\end{equation*}
Each of them is interpreted as an 
oriented rooted planar tree 
with four external edges, one internal edge, four 
external vertices, and two internal vertices in the following manner:

\noindent(i) Three external vertices are labelled $\nabla$;
these correspond to the three tensor factors $\nabla$ in the above 
constituent
\[
\nabla^{\otimes 3}\colon M \otimes M \otimes M  
\longrightarrow
A \otimes A \otimes A;
\]
(ii) one external  vertex---the root vertex---is labelled $\pi$;
this corresponds to the right-most arrow $\pi \colon A \to M$;

\noindent(iii)
the two internal vertices are labelled $\mu$;
in the upper morphism,
these correspond to the arrows labelled $\mu \colon A \otimes A \to A$
and $A \otimes \mu \colon A \otimes A   \otimes A \to A \otimes A$;
in the lower morphism, they correspond
to the arrows labelled $\mu \colon A \otimes A \to A$
and
$\mu \otimes A \colon A \otimes A   \otimes A \to A \otimes A$;

\noindent(iv)
the single
internal edge is labelled $h$;
this corresponds to
the arrow labelled 
\[
A \otimes h \colon  A   \otimes A \longrightarrow A \otimes A
\]
in the upper morphism and labelled
\[ 
h\otimes A \colon  A   \otimes A \longrightarrow A \otimes A
\]
in the lower morphism;

\noindent(v)
the three external edges having
$\nabla$ as a vertex are oriented from the vertices labelled $\nabla$
to the vertices labelled $\mu$;

\noindent(vi) 
the remaining external edge is 
oriented from a vertex labelled $\mu$ to the root vertex labelled $\pi$;

\noindent(vii)
two edges having a vertex labelled $\nabla$ as starting point
meet at a vertex labelled $\mu$, this vertex is joined to the other vertex
labelled $\mu$, oriented in that manner, and it meets the third edge
having a vertex labelled $\nabla$ as starting point at its end point.

\noindent
There are two such oriented rooted planar trees,
one being the mirror image of the other,
and the two composite morphisms spelled out above correspond to these
two labelled oriented rooted planar trees.

Likewise 
the operation $m_4\colon \HHH^{\otimes 4} \to \HHH$
is the sum of three composite morphisms of a similar nature
which, in the language of twisting cochains,
arise from the three constituents
$\tau^1\cup \tau^3$, $\tau^2\cup \tau^2$, and $\tau^3\cup \tau^1$,
cf. \eqref{proc333} above; each such composite morphisms
can be encoded in
the appropriate labelled oriented rooted planar tree.

Formalizing this procedure one arrives at the
description of the operations $m_j$ in terms of sums over
labelled oriented rooted planar trees worked out in detail in
\cite{kontsotw}.

The requisite combinatorics for the construction
in Theorem \ref{alg} 
is provided  by the concept of cofree coalgebra;
likewise in Theorem \ref{c} the necessary combinatorics
is provided  by the concept of free algebra.
The machinery of
labelled oriented rooted planar trees yield an alternate description of the
requisite combinatorial tool.

This discussion so far refers to
ordinary strict algebra (or coalgebra) structures
on the larger object coming into play in the corresponding
contraction.
Instead of the original contraction \eqref{contalg}
where the constituent $A$ is an ordinary augmented differential
graded algebra, 
consider now a contraction of the kind 
\eqref{contalginf} where $A$ is
merely an $A_{\infty}$-algebra,
with structure maps
\[
\mu_j\colon A^{\otimes j} \longrightarrow A\ (j \geq 1).
\]
The construction
in \cite{kontsotw} in terms of
labelled oriented rooted planar trees
extends to that situation. Indeed, 
define the {\em arity\/} of a vertex to be the number of incoming edges.
To the trees
described above having only vertices of arity 2, 
one simply adds trees where internal vertices
$v$ have general arities $\jj > 2$, a vertex of arity $\jj > 2$ being
labelled by the operation $\mu_{\jj}$; one then takes the sum over
all labelled oriented rooted planar trees.
This kind of construction yields an $A_{\infty}$-algebra structure
on $\HHH$ and, extended  suitably,
it also yields an  $A_{\infty}$-equivalence
between $A$ and $\HHH$.

However, unravelling the perturbation $\mathcal D$ 
on $\mathrm T^{\mathrm c}[s \III \HHH]$
spelled out in
Theorem \ref{2.2ainf}, we find precisely that very same
$A_{\infty}$-structure as that given by the
labelled oriented rooted planar trees. Indeed, the infinite series
\eqref{1.1.1}, evaluated relative to the contraction
\eqref{2.2.0c},
takes the form
\begin{equation}
\dell =  
(\mathrm T^{\mathrm c}\ppi)    \partial 
(\mathrm T^{\mathrm c}\nabla)
- (\mathrm T^{\mathrm c} \ppi) \partial
(\mathrm T^{\mathrm c} h)\partial 
(\mathrm T^{\mathrm c}\nabla) +
(\mathrm T^{\mathrm c}\ppi)    \partial 
(\mathrm T^{\mathrm c} h)      \partial 
(\mathrm T^{\mathrm c} h)      \partial
(\mathrm T^{\mathrm c}\nabla) + \ldots.
\label{trees}
\end{equation}
Now, when $\partial$ arises from an ordinary associative differential
graded algebra structure on $A$,
the first term in the development \eqref{trees}
yields, on $\HHH$,
the $A_{\infty}$-constituent $m_2$,
the second term yields the $A_{\infty}$-constituent $m_3$,
and so forth, precisely in the form given by the
labelled oriented rooted planar trees construction.
More general, when $\partial$ arises from a general
$A_{\infty}$-algebra structure 
$\{\mu_j\}_j$ on $A$,
the perturbation
$\partial$ on $\mathrm T^{\mathrm c}[s \III A]$ has the form
\[
\partial = \partial^1 + \partial^2 + \ldots
\]
in such a way that, for $j \geq 1$, the constituent $\partial^j$
corresponds to $\mu_{j+1}$.
Consequently each term in the development \eqref{trees}
involves all the constituents $\mu_j$,
and reordering the terms that show up in
\eqref{trees},
we obtain precisely the $A_{\infty}$-constituent $m_j$
in the form given by the
labelled oriented rooted planar trees construction
for the transference of a general $A_{\infty}$-algebra structure.
Likewise, exploiting the series \eqref{1.1.2},
\eqref{1.1.3}, and \eqref{1.1.4}
to unravel the other terms, respectively, 
$\mathrm T^{\mathrm c}_{\partial}\nabla$,
$\mathrm T^{\mathrm c}_{\partial}\ppi$,
and $\mathrm T^{\mathrm c}_{\partial} h $
that are spelled out in
Theorem \ref{2.2ainf},
we obtain the requisite remaining data
that establish the necessary chain equivalence,
precisely in the form given by the
corresponding labelled oriented rooted planar trees constructions
for the transference of a general $A_{\infty}$-algebra structure.

The same kind of remark applies to the dual situation
encapsulated in Theorem \ref{2.2cinf}.
The construction of these perturbations $\mathcal D$ 
on the tensor algebra or tensor coalgebra
relies on the {\em tensor trick\/}
mentioned above, which we developed in 
\cite{homotype}.

Thus we see that
the more recent constructions 
of an $A_{\infty}$-structure
in \cite{kontsotw} (6.4) and \cite{merkutwo}
still come down to the earlier constructions.

\begin{Remark}\label{merkulov}
{\rm 
The construction in  {\rm \cite{merkutwo}}
is slightly more general in the sense that the initial data
considered there are required to satisfy  
requirements somewhat weaker than those which characterize
an ordinary contraction.
Indeed, in {\rm \cite{merkutwo}},
a system
\begin{equation}
   \left(N
     \begin{CD}
      \null @>{\nabla}>> \null\\[-3.2ex]
      \null @<<{\pi}< \null
     \end{CD}
    M, h\right) \label{coo}
  \end{equation}
of chain complexes is explored satisfying the requirements
\eqref{co1} and \eqref{side} but not necessarily \eqref{co0},
that is, it is {\em not\/} required that
$\pi \nabla = \mathrm{Id}$; indeed, no condition is imposed
upon $\pi \nabla$.
We will now show that\/}
application of the constructions in the perturbation lemma
without the requirement
that
$\pi \nabla = \mathrm{Id}$ does not
lead to a more general theory.

{\rm Indeed, a system of the kind {\rm \eqref{coo}} 
can arise by the following specific construction
from an ordinary contraction
and in fact every system of the kind {\rm \eqref{coo}} arises in this way:
Consider an ordinary contraction
\begin{equation}
   \left(N_1
     \begin{CD}
      \null @>{\nabla_1}>> \null\\[-3.2ex]
      \null @<<{\pi_1}< \null
     \end{CD}
    M, h\right) \label{coo1}
  \end{equation}
of chain complexes, let $N_2$ be an arbitrary chain complex,
let $N=N_1\oplus N_2$, let $\pi\colon M \to N$ be the
composite of $\pi_1$ with the canonical injection into $N$, 
let 
\[
\nabla_2\colon N_2 \to \mathrm{ker}(\pi)=\mathrm{ker}(\pi_1) \subseteq M
\]
be a chain map,
and let
$\nabla=(\nabla_1,\nabla_2)\colon N_1\oplus N_2 \to M$. Then
\begin{equation}
   \left(N
     \begin{CD}
      \null @>{\nabla}>> \null\\[-3.2ex]
      \null @<<{\pi}< \null
     \end{CD}
    M, h\right) \label{coo2}
  \end{equation}
is a system of the kind \eqref{coo}.
We will now show that\/} every system of the kind {\rm \eqref{coo}}
arises in this way. 
Given a perturbation of the differential on $M$,
application of the perturbation lemma
involves only the summand  $N_1$ of $N$
and leaves $N_2$ unchanged
in the sense that this application
yields a new system of the kind
{\rm \eqref{coo2}} where the new contraction of the kind {\rm \eqref{coo1}}
arises by an application of the perturbation lemma
to the contraction of the kind {\rm \eqref{coo1}} before
application of the perturbation and where the summand $N_2$ remains unchanged.
Hence application of the constructions in the perturbation lemma
without the requirement
that
$\pi \nabla = \mathrm{Id}$ does not
lead to a more general theory.

{\rm Thus consider a system of the kind \eqref{coo}.
The requirement \eqref{co1}, viz.
\[
Dh = \mathrm{Id} -\nabla \pi,
\]
implies
\[
D(\pi h\nabla) = \pi \nabla -\pi\nabla \pi\nabla
\]
which, in view of the annihilation properties \eqref{side}, comes down to
$\pi \nabla =\pi\nabla \pi\nabla$.
Hence the endomorphism $\pi \nabla$ of $N$
is a projector. Let $N_1= \pi \nabla(N)$ and $N_2= (\mathrm{Id}-\pi \nabla)(N)$.
Then 
$N=N_1\oplus N_2$.
Let $\pi_1\colon M \to N_1$ be the
composite of $\pi$ with the canonical projection to $N_1$, and let
$\nabla_1\colon N_1 \to M$ be the injection of $N_1$ into $N$, followed by
$\nabla$.}
The resulting data
\begin{equation}
   \left(N_1
     \begin{CD}
      \null @>{\nabla_1}>> \null\\[-3.2ex]
     \null @<<{\pi_1}< \null
    \end{CD}
    M, h\right) \label{coo3}
 \end{equation}
constitute a contraction of chain complexes.

{\rm 
Indeed, to understand the situation, suppose momentarily that
$N_1$ is zero, that is, the composite $\pi \nabla$ is zero.
Then $\nabla \pi= \nabla \pi \nabla \pi = 0$ whence
$Dh = \mathrm{Id}$, that is, $M$ is contractible, the homotopy $h$ being a 
conical contracting homotopy.
In the general case 
where $N_1$ is not necessarily zero,
let $M_1=\nabla(N_1)$ and $M_2=\mathrm{ker}(\pi)$.
The requirement \eqref{co1}, viz.
$Dh = \mathrm{Id} -\nabla \pi$, together with the annihilation properties
\eqref{side}, implies that
\[
M =\mathrm{ker}(\pi) + \nabla \pi (M) =\mathrm{ker}(\pi) + \nabla (N)=
\mathrm{ker}(\pi) + \nabla (N_1)+  \nabla (N_2).
\]
But $\pi \nabla(N_2)$ is zero whence
$\nabla (N_2) \subseteq \mathrm{ker}(\pi)$ and thence
\[
M =\mathrm{ker}(\pi) + \nabla  (N_1) = M_1 + \mathrm{ker}(\pi)= M_1 + M_2.
\]
By construction,  $\pi$ restricted to $M_2$ is zero
and
$\pi$ restricted to $M_1$ amounts to the restriction
of $\pi_1$ to $M_1$ which, in turn is an isomorphism having $\nabla_1$
as its inverse. Moreover, exploiting once more the fact that
$\pi \nabla(N_2)$ is zero we conclude that
$\nabla$ restricted to $N_2$ is a morphism of the kind
$\nabla_2\colon N_2 \to M_2$. 
Hence the decomposition $M=M_1+M_2$ is a direct sum decomposition;
the obvious inclusion 
$\mathrm{ker}(\pi_1)\subseteq \mathrm{ker}(\pi)$ is the identity;
the 
original system \eqref{coo} can be written as
\begin{equation}
   \left(N_1\oplus N_2
     \begin{CD}
      \null @>{(\nabla_1,\nabla_2)}>> \null\\[-3.2ex]
     \null @<<{(\pi_1,0)}< \null
    \end{CD}
     M_1\oplus M_2, h\right) ; \label{coo5}
 \end{equation}
and the annihilation properties \eqref{side} imply that
$h(M_1)$ is zero and that $h(M_2) \subseteq M_2$.
Indeed, since $h\nabla$ is zero, $h$ vanishes on $M_1$ and, likewise,
since
$\pi h$ is zero, $h(M_2) \subseteq M_2$.
Hence \eqref{coo5} actually takes the form
\begin{equation}
   \left(N_1\oplus N_2
     \begin{CD}
      \null @>{(\nabla_1,\nabla_2)}>> \null\\[-3.2ex]
     \null @<<{(\pi_1,0)}< \null
    \end{CD}
     M_1\oplus M_2, (0,h_2)\right)  \label{coo6}
 \end{equation}
and the system
\begin{equation}
   \left(N_2
     \begin{CD}
      \null @>{\nabla_2}>> \null\\[-3.2ex]
     \null @<<{0}< \null
    \end{CD}
     M_2, (0,h_2)\right)  \label{coo7}
 \end{equation}
is of the kind of the special case where $N_1$ is zero---indeed,
the corresponding morphism $\pi_1$ is now even zero---whence, in particular,
$h_2$ is a conical contracting homotopy for $M_2$.
Consequently
\begin{equation}
   \left(N_1
     \begin{CD}
      \null @>{(\nabla_1,0)}>> \null\\[-3.2ex]
     \null @<<{(\pi_1,0)}< \null
    \end{CD}
     M_1\oplus M_2, (0,h_2)\right) , \label{coo8}
 \end{equation}
is an ordinary contraction as asserted and the original system \eqref{coo}
is indeed of the special kind \eqref{coo2}.

}
\end{Remark}
\section{Perturbations for Lie algebras and 
$L_{\infty}$-algebras, and the master equation}

Let $\mathfrak g$  be a differential graded $R$-Lie algebra
that is projective as a graded $R$-module
 and let
\begin{equation}
   (\HHH
     \begin{CD}
      \null @>{\nabla}>> \null\\[-3.2ex]
      \null @<<{\pi}< \null
     \end{CD}
    \mathfrak g, h)
\label{contlie}
  \end{equation}
be a contraction of {\em chain complexes\/}.
Suppose that
the cofree coaugmented differential graded cocommutative
coalgebra $\Sigm^{\mathrm c}[s\fra g]$ on the suspension $s\fra g$
of $\fra g$ and, likewise,
the cofree coaugmented differential graded cocommutative
coalgebra
$\Sigm^{\mathrm c} = \Sigm^{\mathrm c}[s\HHH]$
on the suspension $s\HHH$ of $\HHH$ exist.
This kind of coalgebra is well
known to be cocomplete.
Further, let
$\dd^0\colon \Sigm^{\mathrm c} \longrightarrow \Sigm^{\mathrm c}$
denote the coalgebra differential on $\Sigm^{\mathrm c} =
\Sigm^{\mathrm c}[s\HHH]$ induced by the differential on $\HHH$. For $b
\geq 0$, we will henceforth denote the homogeneous degree $b$
component of $\Sigm^{\mathrm c}[s\HHH]$ by $\Sigm_b^{\mathrm c}$;
thus, as a chain complex, $\mathrm F_b\Sigm^{\mathrm c} = R \oplus
\Sigm_1^{\mathrm c} \oplus \dots \oplus \Sigm_b^{\mathrm c}$.
Likewise, as a chain complex, $\Sigm^{\mathrm c} =
\oplus_{j=0}^{\infty} \Sigm_j^{\mathrm c}$. We denote by
\[
\tau_{\HHH}\colon \Sigm^{\mathrm c} \longrightarrow \HHH
\]
the composite of the canonical projection $\mathrm{proj}\colon
\Sigm^{\mathrm c} \to s\HHH$ from $\Sigm^{\mathrm c} = \Sigm^{\mathrm
c}[s\HHH]$ to its homogeneous degree 1 constituent $s\HHH$ with the
desuspension map $s^{-1}$ from $s\HHH$ to $\HHH$.
When $\HHH$ is viewed as an {\em
abelian\/} differential graded Lie algebra, $\Sigm^{\mathrm
c}=\Sigm^{\mathrm c}[s\HHH]$ may be viewed as the {\scs CCE\/} or
{\em classifying\/} coalgebra $\mathcal C [\HHH]$ for $\HHH$, and
$\tau_{\HHH} \colon \Sigm^{\mathrm c} \to \HHH$ is then the universal
differential graded Lie algebra twisting cochain for $\HHH$.

Let
\begin{equation}
\tau^1 = \nabla \tau_{\HHH}\colon  
\Sigm^{\mathrm c} \to \fra g
\label{proc00001}
\end{equation}
and, for $\jj \geq 2$, 
let
\[
\tau^j \colon  
\Sigm^{\mathrm c}
\to \fra g
\]
be the degree $-1$ morphism defined recursively by
\begin{equation}
\tau^{\jj} = 
 \tfrac 12 h([\tau^1,\tau^{{\jj}-1}] +  \dots +
[\tau^{{\jj}-1},\tau^1])\colon\Sigm^{\mathrm c} \to \fra g.
\label{proc0011}
\end{equation}
Thereafter, for $\jj \geq 1$,
define the degree $-1$ coderivation
$\mathcal D^j$ on 
$\Sigm^{\mathrm c}$
by
\begin{equation}
\tau_{\HHH} \mathcal D^{\jj} =  
\tfrac 12 \pi ([\tau^1,\tau^{\jj}] +
\dots + [\tau^{\jj},\tau^1]) \colon \Sigm_{\jj+ 1}^{\mathrm c} \to
\HHH. \label{proc00333}
\end{equation}
In particular, for $\jj \geq 1$, the coderivation $\mathcal
D^{\jj}$ is zero on $\mathrm F_j\Sigm^{\mathrm c} $ 
and lowers
coaugmentation filtration by $\jj$.

We now spell out the main result of \cite{pertlie}.

\begin{Theorem}\label{liealg}
The infinite series
\begin{equation}
\mathcal D = \mathcal D^1 + \mathcal D^2 + \dots \colon
\Sigm^{\mathrm c} \to \Sigm^{\mathrm c}
 \label{proc02111}
\end{equation}
is a coalgebra perturbation of the differential $d$ on
$\Sigm^{\mathrm c}$
induced from the differential on $\HHH$, and the infinite series
\begin{equation}
\tau = \tau^1 + \tau^2 + \dots \colon 
\Sigm^{\mathrm c} \to \fra g
\label{proc11111}
\end{equation}
is a Lie algebra twisting cochain
\[
(\Sigm^{\mathrm c}, d + \mathcal D)
\longrightarrow \fra g.
\]
Furthermore, the adjoint 
\[
\overline \tau \colon(\Sigm^{\mathrm c}, d + \mathcal D)
\longrightarrow 
\mathcal C \fra g
\]
of the twisting cochain $\tau$, necessarily 
a morphism of differential graded coalgebras, is a chain equivalence.
More precisely, the data determine a
contraction
\begin{equation*}
   \left((\Sigm^{\mathrm c}[s\HHH],d + \ppartial)
     \begin{CD}
      \null @>{\overline \tau}>> \null\\[-3.2ex]
      \null @<<{\Pi}< \null
    \end{CD}
   \mathcal C[\fra g], H\right)
 \end{equation*}
of chain complexes which is natural in terms of the data.
\end{Theorem}

This is precisely Theorem 2.1 in \cite{pertlie}
where also a complete proof can be found.
For the special case where $\HHH$ is the homology of $\fra g$,
this result yields the 
\lq\lq minimality theorem\rq\rq\ for ordinary differential 
graded Lie algebras.

\begin{Remark}\label{mastereq}
The attempt to treat,
as for the cases explained in Section {\rm \ref{algebra}} above,
the requisite higher
homotopies by means of a suitable version of the perturbation lemma 
relative to the
{\em additional algebraic structure\/}, that is, to develop a version of
the perturbation lemma {\em compatible with Lie brackets or more
generally with sh-Lie structures\/}, led to the paper
{\rm\cite{huebstas}}, but technical complications arise since
the {\em tensor trick\/}
breaks down for cocommutative
coalgebras; indeed, the notion of {\em homotopy of morphisms of
cocommutative coalgebras is a subtle concept\/} {\rm\cite{schlstas}},
and only a special case was handled in {\rm\cite{huebstas}}, with some
of the technical details merely sketched. The 
article {\rm \cite{pertlie}}
provides 
a complete solution with
{\em all\/} the necessary details and handles the case of
a {\em general contraction\/} whereas in {\rm \cite{huebstas}} only the
case of a contraction of a differential graded Lie algebra onto
its {\em homology\/} was treated.
Also in {\rm \cite{huebstas}}, the proof is only sketched, and
a detailed proof is given in {\rm \cite{pertlie}}.
The twisting cochain $\tau$ is the most general solution of the 
{\em master equation\/}, under the circumstances of {\rm \cite{pertlie}}.
\end{Remark}

We will  write
\begin{equation}
\mathcal C_{\mathcal D}\HHH=
(\Sigm^{\mathrm c}, d + \mathcal D) .
\label{shl}
\end{equation}
The piece of structure
$\mathcal D$ in
$\mathcal C_{\mathcal D}\HHH$ is precisely an
$L_{\infty}$-algebra structure on $\HHH$,
Theorem \ref{liealg} includes the statement that
$\HHH$, endowed with the $L_{\infty}$-algebra structure
$\mathcal D$, and $\fra g$, endowed with the 
$L_{\infty}$-algebra structure associated with the 
differential graded Lie algebra structure,
are, via $\tau$, equivalent as $L_{\infty}$-algebras.

The general sh-Lie algebra perturbation lemma,
Theorem 2.5 in \cite{pertlie2},
extends 
Theorem \ref{liealg} to the more general case where the constituent
$\fra g$ in the contraction \eqref{contlie}
is merely an sh-Lie algebra. 
This sh-Lie algebra perturbation lemma yields, in particular,
the \lq\lq minimality theorem\rq\rq\ for general
$L_{\infty}$-algebras.

\section{More about the relationship with deformations}

We have pointed out above that
the idea of combining the Gerstenhaber operation in the Hochschild complex
mentioned in Section \ref{kadea}
with Stasheff's notion of $A_\infty$-algebra let 
Kadeishvili to the minimality theorem.
There is an obvious formal relationship between homological
perturbations and deformation theory but the relationship is
actually much more profound: In \cite{halpsttw}, S.~Halperin
and J.~Stasheff developed a procedure by means of which the
classification of rational homotopy equivalences inducing a fixed
cohomology algebra isomorphism can be achieved. Moreover, one can
explore the rational homotopy types with a fixed cohomology
algebra by studying perturbations of a free differential graded
commutative model by means of techniques from deformation theory.
This was initiated by M. Schlessinger and J. Stasheff
\cite{schlstas}. A related and independent development, phrased in
terms of what is called the functor $\mathcal D$, is due to N.
Berikashvili and his students at Georgia, notably T. Kadeishvili
and S. Saneblidze; see 
\cite{berikash} for some details and references.
A third approach in which only the underlying graded vector space was
fixed is due to Y. Felix \cite{felix}.
More remarks about the relationship between
homological
perturbations and deformation theory can be found in \cite{origins}.

Kadeishvili contributed once more to the relationship between
deformations and higher homotopies:
He observed 
an interpretation of $A_\infty$-operations in terms of a suitable notion of twisting
cochain, with respect to the Gerstenhaber operation in the Hochschild
(cochain) complex \cite{kadeisix}. An expanded version of that approach
will appear as \cite{kadeieig}.

\end{document}